\documentclass{article}

\usepackage[dvipdfmx]{graphicx}
\usepackage[usenames]{color}
\usepackage{bm}
\usepackage{amsmath,amssymb,amsthm}
\usepackage{tabularx}
\usepackage{lscape}
\usepackage{url}
\usepackage[normalem]{ulem}
\usepackage{microtype}
\usepackage{subfigure}
\usepackage{booktabs} 
\usepackage{enumerate}
\usepackage{a4wide}
\usepackage[bottom=30truemm]{geometry}
\usepackage{algorithm,algorithmic}
\usepackage[numbers]{natbib}
\usepackage{mathtools}

\mathtoolsset{showonlyrefs,showmanualtags}
\usepackage{amsfonts}

\newcommand{\argmin}{\mathop{\rm argmin}\limits}

\newtheorem{theorem}{Theorem}
\newtheorem{corollary}{Corollary}
\newtheorem{lemma}{Lemma}
\newtheorem{definition}{Definition}
\newtheorem{remark}{Remark}
\newtheorem{example}{Example}
\newtheorem{proposition}{Proposition}

\title{Pursuit of the Cluster Structure of Network Lasso:\\
Recovery Condition and Non-convex Extension}
\author{Shotaro Yagishita\footnote{Department of Industrial and Systems Engineering, Chuo University, Japan ({a15.fjng@g.chuo-u.ac.jp}, {jgoto@indsys.chuo-u.ac.jp}).}
\and Jun-ya Gotoh\footnotemark[2]}

\begin{document}
\maketitle
\begin{abstract}
Network Lasso (NL for short) is a methodology for estimating models by simultaneously clustering data samples and fitting the models to the samples. 
It often succeeds in forming clusters thanks to the geometry of the $\ell_1$-regularizer employed therein, but there might be limitations because of the convexity of the regularizer. This paper focuses on the cluster structure that NL yields and reinforces it by developing a non-convex extension, which we call Network Trimmed Lasso (NTL for short). Specifically, we first study a sufficient condition that guarantees the recovery of the latent cluster structure of NL on the basis of the result of Sun et al.~(2021) for Convex Clustering, which is a special case of NL for clustering. Second, we extend NL to NTL to incorporate a cardinality (or, $\ell_0$-)constraint and rewrite the constrained optimization problem defined with the $\ell_0$ norm, a discontinuous function, into an equivalent unconstrained continuous optimization problem. We develop ADMM algorithms to solve NTL and provide its convergence results. Numerical illustrations demonstrate that the non-convex extension provides a more clear-cut cluster structure when NL fails to form clusters without incorporating prior knowledge of the associated parameters.
\end{abstract}

\section{Introduction}
In data analysis, fundamental methodologies such as regression and clustering can be enhanced by coupling with side information about the underlying structure of the data set. For example, consider a situation where a batch of data samples consists of outcomes from multiple sources and the analyst knows their (possibly, partial) relationship. In such situations, we can estimate multiple models to fit to the dataset while detecting clusters of the samples. To accomplish such tasks, \citet{hallac2015network} recently have proposed \emph{Network Lasso} (NL for short). 

Let $a_i\in\mathbb{R}^p$ and $b_i\in\mathbb{R}$ denote $p$ inputs and a real-valued output, respectively, of the $i$-th sample, $i\in[n]:=\{1,...,n\}$, and assume that the analyst knows that some of them are similar. If such similarity for $i,j\in[n]$ is given by non-negative weights, $\tilde{w}_{\{i,j\}}$, its NL version of the ordinary least squares regression is cast as the following convex optimization:
\begin{align}
\underset{x_1,...,x_n\in\mathbb{R}^p}{\text{minimize}}&\quad
\frac{1}{2}\sum_{i\in[n]}(b_i-a_i^\top x_i)^2+\gamma\sum_{i,j\in[n]:i<j}\tilde{w}_{\{i,j\}}\|x_i-x_j\|_2,
\label{eq:nw_regress}
\end{align}
where $\|z\|_2:=\sqrt{\sum_{j=1}^{p}z_j^2}$ denotes the $\ell_2$ norm of a vector $z\in\mathbb{R}^p$, and $\gamma>0$ is a parameter to be tuned so as to strike a balance between the first and second terms of \eqref{eq:nw_regress}. 
Intuitively, reducing the first term prompts a model ($b=a^\top x_i$) to fit each sample $(a_i,b_i)$, while reducing the second term expedites the mergers of similar samples since the weight $\tilde{w}_{\{i,j\}}$ being large means that the samples $i$ and $j$ are similar to each other.
Especially, the second term is the sum of $\ell_2$ norms and, for large $\gamma$, an optimal solution $(x^\star_1,...,x^\star_n)$ is expected to satisfy $\|x^\star_i-x^\star_j\|_2=0$ for many pairs $\{i,j\}\in\mathcal{E}$. 
This contraction property is parallel to the group Lasso \citep{yuan2006model}.

In general, let us introduce a weighted undirected graph $\mathcal{G}=(\mathcal{V},\mathcal{E},W)$, where the vertex set $\mathcal{V}=[n]$ denotes the index set of samples, the edge set $\mathcal{E}\subset\{\{i,j\}: i,j\in[n];\,i\neq j\}$  indicates the pairwise adjacency or similarity, and $W=(w_{\{i,j\}})_{\{i,j\}\in\mathcal{E}}\in\mathbb{R}^{\left|\mathcal{E}\right|}_{\geq0}$ denotes non-negative weights on all the edges to represent the pairwise similarity. (The more $w_{\{i,j\}}$, the closer the vertices $i$ and $j$.) 
 Let $f_{i}:\mathbb{R}^{p}\rightarrow\mathbb{R}$ be a loss function for sample $i\in[n]$. 
NL \citep{hallac2015network} is then formulated as the following optimization problem:
\begin{align}
\underset{x_1,\ldots,x_n\in\mathbb{R}^{p}}{\text{minimize}}
 &\quad \sum_{i\in\mathcal{V}}f_i(x_i)+\gamma\sum_{\{i,j\}\in\mathcal{E}}w_{\{i,j\}}\|x_i-x_j\|_{2}.
\label{nlasso}
\end{align}
Obviously, \eqref{eq:nw_regress} is an example of NL \eqref{nlasso}, where the sum of squared residuals, $f_i(x_i)=\frac{1}{2}(b_i-a_i^\top x_i)^2$, $i\in[n]$, are employed as the loss functions. 
NL includes other methods as special cases.
If only the input vectors $a_i\in\mathbb{R}^p, i\in[n]$, are given and we employ
\begin{equation}
f_i(x_i)=\frac{1}{2}\|x_i-a_i\|_2^2,
\label{convexclustering}
\end{equation}
and set $\mathcal{E}=\{\{i,j\}: i,j\in[n];\,i\neq j\}$ and $w_{\{i,j\}}=1$, NL \eqref{nlasso} is reduced to \emph{Convex Clustering} (\citet{pelckmans2005convex}). \citet{lindsten2011clustering} and \citet{hocking2011clusterpath} consider extensions where $w_{\{i,j\}}$ are not necessarily equal to $1$. 
With an optimal solution $(x_1^*,...,x_n^*)$, nodes $i$ and $j$ are assigned to the same cluster if and only if $x_i^*=x_j^*$. We call $x_i^*$ the {\it centroid} of node $i$. Namely, samples that share the same centroid form a cluster. 
For sufficiently large $\gamma$, a clustering result of the data set $a_1,...,a_n$ is obtained thanks to the exact contraction property of the second term of \eqref{nlasso}. 
Recent works have reported that NL numerically performs well in various tasks when prior information about the similarity between nodes (i.e., $W$) are given appropriately (e.g., \citet{hallac2015network,jung2018when,hocking2011clusterpath,chi2015splitting,sun2021convex}).

In this paper, we further investigate and extend NL with a focus on its clustering property.
First, we study whether NL can recover true latent clusters.
For Convex Clustering, \citet{zhu2014convex,panahi2017clustering} and \citet{sun2021convex} provide sufficient conditions for recovering the set of the latent clusters.
For NL (not limited to Convex Clustering), although \citet{jung2018when} and \citet{jung2019localized} analyze the gap between the optimal solution of NL and true parameter values. However, their conditions do not guarantee the recovery of the true clusters. 
In contrast, we provide sufficient conditions to recover the latent clusters for NL.
The easiness of the recovery will be given by ranges of the parameter $\gamma$, for which NL \eqref{nlasso} recovers the (unseen) true clusters if they exist. 

Besides, it is almost certain that there is a gap between Convex Clustering and the usual $k$-means approach (or its non-convex optimization version). 
The second part of this paper is devoted to building a bridge between the two realms: convex vs. non-convex. 
We should emphasize here that the performance of \eqref{nlasso} highly depends on how the analyst puts the prior information of the weights $W$.
The left panel of Figure \ref{visual1} shows the regularization paths, i.e., the loci of the centroids obtained by Convex Clustering without prior information (i.e., $w_{(i,j)}=1$ for all $(i,j)\in\mathcal{E}$). 
While there are two latent clusters (red and blue), all the centroids shrink to the middle point in an equal manner and we cannot obtain the two clusters even with a big $\gamma$. 
For a practical use of Convex Clustering, it is often suggested to set the weights $w_{\{i,j\}}$, as $w_{\{i,j\}}=\exp(-\alpha\|a_i-a_j\|_2^2)$, where $\alpha>0$ is a parameter. 
The right panel of Figure \ref{visual1} demonstrates that, with this technique, Convex Clustering resulted in a clear-cut clustering. 
We should note that how to provide such prior information depends on tasks and there are no general tips for NL \eqref{nlasso} (e.g., for regression). 

To overcome such difficulty, we consider introducing a cardinality constraint in place of the group $l_2$-penalty in NL \eqref{nlasso}.
We show that the cardinality-constrained problem can be equivalently rewritten by a non-convex but continuous unconstrained optimization problem, which we call \emph{Network Trimmed Lasso} (NTL for short). 
This reformulation is parallel to that for the \emph{Trimmed Lasso}, which is studied by, for example,  \citet{gotoh2018dc,bertsimas2017trimmed}, and \citet{amir2020trimmed}. 
We also propose algorithms based on Alternating Direction Method of Multipliers (ADMM) to solve NTL. 
For a non-convex subproblem in the proposed algorithms, a closed-form solution is derived.
Also, we show the convergence of Proximal ADMM, which is an extension of ADMM, to a locally optimal solution of NTL, a non-convex optimization problem.
The advantages of NTL over ordinary NL are demonstrated through numerical experiments.

 \begin{figure}[H]
  \begin{tabular}{cc}
   \begin{minipage}[t]{0.45\hsize}
    \centering
    \includegraphics[width=1.0\columnwidth]{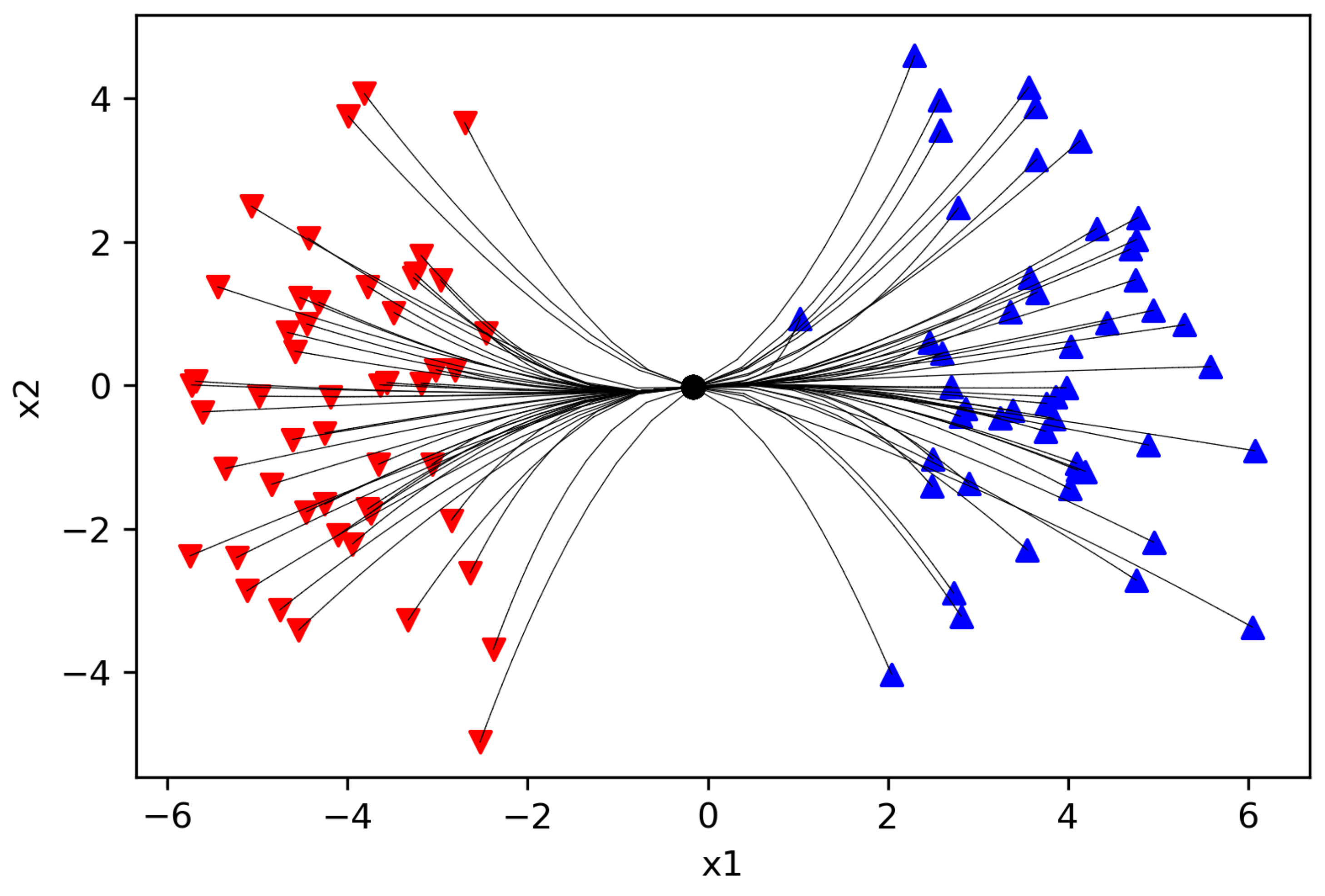}
   \end{minipage} &
   \begin{minipage}[t]{0.45\hsize}
    \centering
    \includegraphics[width=1.0\columnwidth]{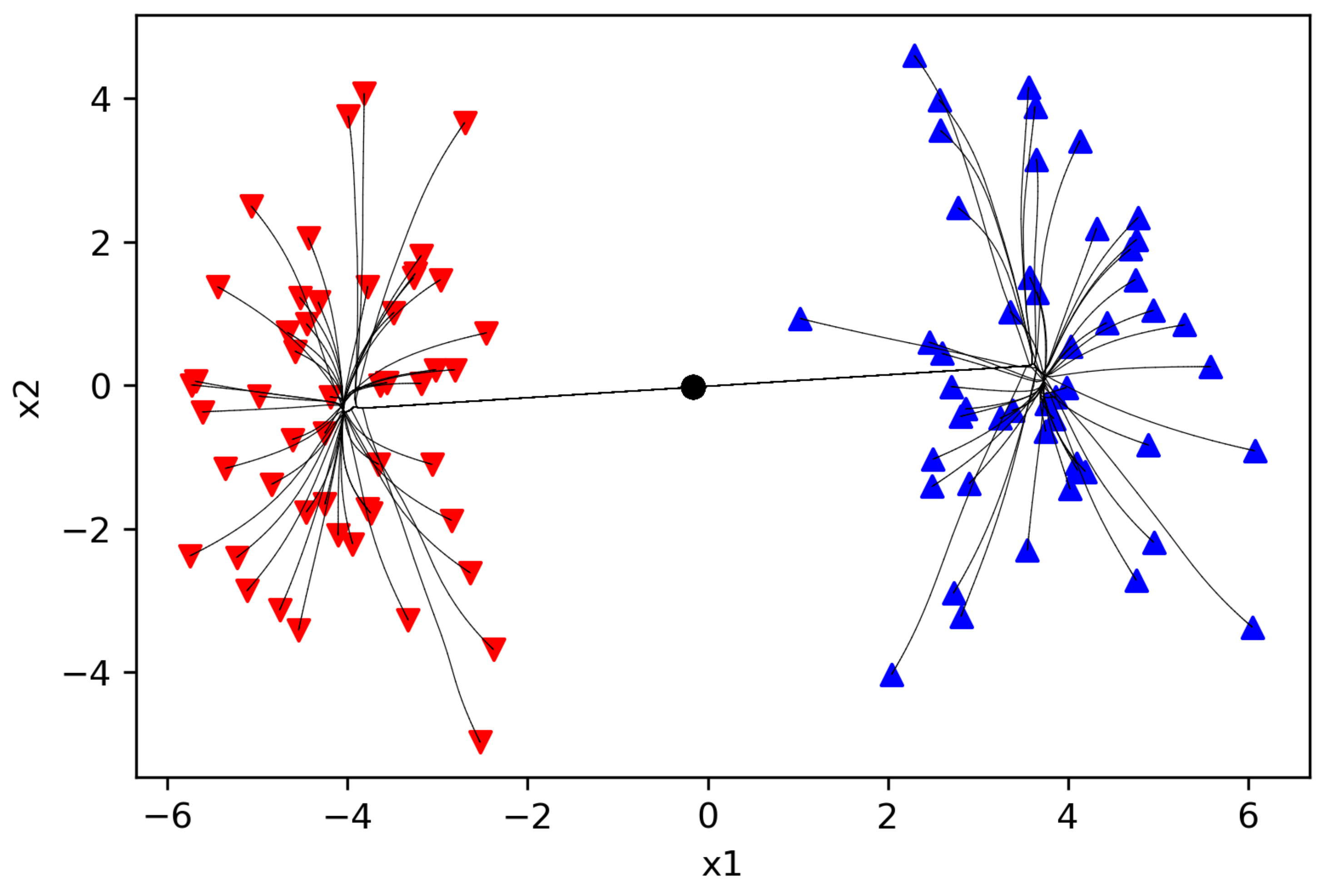}
   \end{minipage}\\
(a) Without prior information & (b) With prior information
  \end{tabular}
  \caption{Failed cluster path of centroids via Convex Clustering without prior information (left panel (a)) and successful cluster path with prior information, where $w_{\{i,j\}}=\mathrm{e}^{-0.1{\|a_i-a_j\|}_2^2}$ (right panel (b)).}
  \label{visual1}
\begin{flushleft}\scriptsize
As the regularization parameter $\gamma$ grows, all the centroids are reduced to the mean of $n$ points, i.e., $\frac{1}{n}\sum_{i\in\mathcal{V}}a_i$, which is located near the origin in each picture. The left-hand side panel shows the result of Convex Clustering with $w_{\{i,j\}}=1$, failing to form clusters even for large $\gamma$'s. On the other hand, the right-hand side panel shows the case where the distance of points is used succeeded in providing a clear-cut cluster structure even with small $\gamma$'s.
\end{flushleft}
 \end{figure}

Contributions of the paper are summarized as follows:
\begin{itemize}
\item We give a theoretical guarantee that the clustering based on NL works well when prior information is properly provided. Specifically, sufficient conditions with which NL can recover the latent cluster structure are given. 
\item We propose a framework of NTL to enhance the cluster structure of NL.
More specifically, we show the equivalence between the cardinality-constrained extension of NL and NTL. 
\item We present ADMM-based algorithms to solve a problem that has a general Trimmed Lasso penalty, and show that under mild conditions it converges to a local optimum. 
\end{itemize}

The rest of this paper is organized as follows.
The next section is devoted to show sufficient conditions that NL recovers a latent cluster structure. 
Section \ref{sec:NTL} introduces a cardinality-constrained version of NL to get a more distinguishing cluster structure than NL does. The equivalence between the cardinality-constrained problem and NTL is shown. 
In Section \ref{sec:algorithm}, we develop algorithms for NTL and provide its convergence results.
Section \ref{sec:numerical} reports numerical examples to demonstrate the effectiveness of NTL.
Finally, Section \ref{sec:conclusion} concludes the paper.

\paragraph{Notation and preliminaries}
A continuous differentiable function $f:\mathbb{R}^{p}\to\mathbb{R}$ is said to be \emph{$L$-smooth} if there exists $L>0$ such that
\begin{align}
\|\nabla f(x)-\nabla f(y)\|_2\leq L\|x-y\|_2
\end{align}
for any $x, y\in\mathbb{R}$.
If $f$ is $L$-smooth, then the inequality
\begin{align}
\label{inv_descent_lemma}
f(y)\leq f(x)+\nabla f(x)^\top (y-x)+\frac{L}{2}{\|x-y\|}_2^2
\end{align}
holds for all $x, y\in\mathbb{R}^{p}$, which implies $\frac{L}{2}\|\cdot\|_2^2-f$ is convex (see, e.g., \citep[Lemma 5.7]{beck2017first}).
For a differentiable convex function $f:\mathbb{R}^{p}\to\mathbb{R}$, the following statements are equivalent (\citep[Theorem 5.8]{beck2017first}):
\begin{itemize}
\item $f$ is $L$-smooth.
\item The inequality \eqref{inv_descent_lemma} holds for all $x, y\in\mathbb{R}^{p}$.
\item $\frac{L}{2}\|\cdot\|_2^2-f$ is convex.
\end{itemize}
For a convex function $f:\mathbb{R}^{p}\to\mathbb{R}$, the \emph{subdifferential} of $f$ at $x\in\mathbb{R}^p$ is defined by
\begin{align}
\partial f(x):=\{z\in\mathbb{R}^p\mid f(y)\ge f(x)+z^\top(y-x) \quad\forall y\in\mathbb{R}^p\}.
\end{align}
Note that $\partial f(x)\neq\emptyset$ (\citep[Theorem 3.14]{beck2017first}).
We call $f:\mathbb{R}^{p}\rightarrow\mathbb{R}$ \emph{strongly convex} with a positive constant $\alpha$ (or simply, \emph{$\alpha$-strongly convex}) if $f-\frac{\alpha}{2}\|\cdot\|_2^2$ is a convex function. 
If $f$ is $\alpha$-strongly convex, by using Theorems 3.63 and 5.24 of \citep{beck2017first} and the Cauchy-Schwarz inequality, we obtain
\begin{align}
\label{l_t_bound}
\|z\|_2\geq\alpha\|x-\overline{x}\|_2,
\end{align}
for all $x\in\mathbb{R}^p,z\in\partial f(x)$, where $\overline{x}\in\argmin_{x\in\mathbb{R}^p}f(x)$ and its existence and uniqueness are guaranteed by strong convexity of $f$ (\citep[Theorem 5.25]{beck2017first}).
It is also known (e.g., \citep[Theorem 5.25]{beck2017first}) that if $f$ is $\alpha$-strongly convex and $\overline{x}\in\argmin_{x\in\mathbb{R}^p}f(x)$, the inequality
\begin{align}
\label{plain_bound}
\frac{\alpha}{2}\|x-\overline{x}\|_2^2\le f(x)-f(\overline{x}),
\end{align}
holds for all $x\in\mathbb{R}^p$.
The \emph{directional derivative} of $f:\mathbb{R}^{p}\to\mathbb{R}$ at a point $x\in\mathbb{R}^p$ in the direction $v\in\mathbb{R}^p$ is defined by
\begin{align}
{\mathrm d}f(x;v):=\lim_{\eta\searrow 0}\frac{f(x+\eta v)-f(x)}{\eta}.
\end{align}
A point $x^*\in\mathbb{R}^p$ is called a \emph{(directional-)stationary point} of an optimization problem $\min\limits_{x} f(x)$ if the directional derivative ${\mathrm d}f(x^*;v)$ exists and is non-negative for any $v\in\mathbb{R}^p$.
The maximum and minimum eigenvalues of a symmetric matrix $A$ are denoted by $\lambda_{\max}(A)$ and $\lambda_{\min}(A)$, respectively.

\section{Recovery conditions for Network Lasso}
\label{sec:recovery}
In this section, we show the recovery conditions for NL \eqref{nlasso} to identify the latent cluster structure on the basis of \citet{sun2021convex}, which develops sufficient conditions for the recovery of the latent cluster structure for Convex Clustering. 

Let $C_1,...,C_N$ denote (unseen) $N$ clusters, which satisfy that $C_i\cap C_j=\emptyset$ if $i\neq j$, and $C_1\cup\cdots\cup C_N=\mathcal{V}$. 
We assume that each sample $i\in\mathcal{V}$ belongs to one of $C_1,...,C_N$. 

To define the recovery of the cluster structure, we introduce a couple of notions, as below, following \citep{sun2021convex}. 
\begin{definition}
Let $\mathcal{P}:=\{{C}_{1},\ldots,{C}_{N}\}$ and $\overline{\mathcal{P}}:=\{\overline{C}_{1},\ldots,\overline{C}_{M}\}$ be partitionings of $\mathcal{V}$.
  \begin{quote}
  \begin{enumerate}
   \item When $\overline{\mathcal{P}}=\mathcal{P}$, we say that $\overline{\mathcal{P}}$ perfectly recovers $\mathcal{P}$.
   \item We call $\overline{\mathcal{P}}$ a coarsening of $\mathcal{P}$ if for any $\overline{C}\in\overline{\mathcal{P}}$ there exists $I\subset\{1,\ldots,N\}$ such that $\overline{C}=\cup_{l\in I}{C}_{l}$. Moreover, $\overline{\mathcal{P}}$ is called the trivial coarsening if $\overline{\mathcal{P}}=\{\mathcal{V}\}$. Otherwise, it is called a non-trivial coarsening.
  \end{enumerate}
  \end{quote}
\end{definition}
A partitioning represents a cluster structure of the data set $\mathcal{V}$. 
In the remainder of this paper, we use $\mathcal{P}$ to refer to the partitioning corresponding to the true (usually, unseen) cluster structure. 
 For the partitioning $\mathcal{P}=\{{C}_{1},\ldots,{C}_{N}\}$, we introduce the following notation:
\begin{align*}
n_k        &:=|C_k|,                                        & k\in[N], & \quad\small\text{(size of Cluster }k\text{)}\\
w_i^{(k)}  &:=\sum_{j\in C_{k}}w_{\{i,j\}},                 & i\in\mathcal{V},k\in[N],& \quad\small\text{(sum of weights of Sample }i\text{ adjacent to Cluster }k\text{)}\\
w^{(k,k')} &:=\sum_{i\in C_k}\sum_{j\in C_{k'}}w_{\{i,j\}}, & k,k'\in[N].& \quad\small\text{(sum of weights between Clusters }k\text{ and }k'\text{)}
\end{align*}
For the sake of simplicity, we set $w_{\{i,j\}}=0$ for $\{i,j\}\notin\mathcal{E}$ in this section.

\begin{theorem}
\label{thm:clusterrecover}
Suppose that $f_i$ is strictly convex and $L_i$-smooth, $i\in\mathcal{V}$. 
Let $\mathcal{P}=\{{C}_{1},\ldots,{C}_{N}\}$ be the (unseen) true partitioning of $\mathcal{V}$, and let $f^{(k)}:=\sum_{i\in{C}_{k}}f_i,~k\in[N]$. 
Assume that for each $k\in[N]$, $f^{(k)}$ is $\alpha_k$-strongly convex, and let ${\overline{x}}^{(k)}=\argmin_{x\in\mathbb{R}^{p}}f^{(k)}(x)$. 
Suppose that ${\overline{x}}^{(k)}\neq {\overline{x}}^{(k')}$ for $k\neq k'$. Let 
\begin{align*}
\mu_{ij}^{(k)}&=\sum_{l\neq k}\left|w_i^{(l)}-w_j^{(l)}\right|+\frac{L_i+L_j}{\alpha_k}\sum_{l\neq k}w^{(k,l)}, &i,j\in C_k,k\in[N],
\end{align*}
and suppose that $n_kw_{\{i,j\}}>\mu_{ij}^{(k)}$ for all $i,j\in C_k,k\in[N]$ s.t. $i\neq j$. Let
\begin{align*}
\gamma_{\max}&:=\min_{k\neq k'}\left\{\frac{{\|{\overline{x}}^{(k)}-{\overline{x}}^{(k')}\|}_2}{\frac{1}{\alpha_k}\sum\limits_{l\neq k}w^{(k,l)}+\frac{1}{\alpha_{k'}}\sum\limits_{l\neq k'}w^{(k',l)}}\right\},\\
\gamma_{\min}&:=\max_{k}\max_{\substack{i,j\in {C}_{k}\\ i\neq j}}\left\{\frac{{\left\|\nabla f_j({\overline{x}}^{(k)})-\nabla f_i({\overline{x}}^{(k)})\right\|}_{2}}{n_kw_{\{i,j\}}-\mu_{ij}^{(k)}}\right\},
\end{align*}
where we set $\frac{a}{0}=\infty$ for $a>0$. 

Let $(x_1^*,...,x_n^*)$ be an optimal solution to \eqref{nlasso}, and $\overline{\mathcal{P}}$ be the quotient set of $\mathcal{V}$ by equivalence relation ${x}_{i}^{*}={x}_{j}^{*}$.
\begin{quote}
\begin{enumerate}
\item If $\gamma_{\min}\leq\gamma<\gamma_{\max}$, $\overline{\mathcal{P}}$ perfectly recovers $\mathcal{P}$.
\item If $\gamma_{\min}\leq\gamma<\max_{k}\frac{{\|\nabla f^{(k)}(\overline{x})\|}_2}{\sum\limits_{l\neq k}w^{(k,l)}}$, $\overline{\mathcal{P}}$ is a non-trivial coarsening of $\mathcal{P}$, where 
\begin{align*}
\overline{x}:=\argmin_{x\in\mathbb{R}^{p}}\sum\limits_{k\in[N]}f^{(k)}(x)=\argmin_{x\in\mathbb{R}^{p}}\sum\limits_{i\in\mathcal{V}}f_i(x).
\end{align*}
\end{enumerate}
\end{quote}
\end{theorem}

Before proving the theorem, let us notice a couple of remarks on its statement.

\begin{remark}
Theorem \ref{thm:clusterrecover} implies that if the weights $(w_{\{i,j\}})_{\{i,j\}\in\mathcal{E}}$ are chosen adequately, NL is guaranteed to return the true cluster structure $\{C_1,...,C_N\}$ at some point on the cluster path. 
To see this through an example, let us suppose that $\mathcal{G}$ is a complete graph $\mathcal{E}=\{\{i,j\}\mid i\neq j,~ i,j\in\mathcal{V}\}$, $\{f_i\}_{i\in\mathcal{V}}$ satisfies the assumption of Theorem \ref{thm:clusterrecover}, and consider the weights defined as
\begin{align}
w_{\{i,j\}}
\left\{
\begin{array}{ll}
\le w, &(i\in {C}_{k},\ j\in {C}_{k'},\ k\neq k'), \\
=1, &(i,j\in {C}_{k},\ i\neq j),
\end{array}
\right.
\end{align}
for a constant $w\in[0,1]$. 
Note that when $w$ is equal to $1$, we can say the weights have no information; 
as $w$ gets closer to $0$, the weights more reflect the true cluster structure. 
Observing that 
\begin{align}
\sum\limits_{l\neq k}w^{(k,l)}\le\sum\limits_{l\neq k}n_kn_lw\le n^2w\rightarrow0\quad(w\rightarrow0)
\end{align}
for all $k\in[N]$ and
\begin{align}
\sum_{l\neq k}\left|w_i^{(l)}-w_j^{(l)}\right|\le\sum_{l\neq k}n_lw\le nw\rightarrow0\quad(w\rightarrow0)
\end{align}
for all $i,j\in C_k,k\in[N]$, we have
\begin{align}
\gamma_{\max} &\rightarrow\infty,\\
\gamma_{\min} &\rightarrow\max_{k}\max_{\substack{i,j\in {C}_{k}\\ i\neq j}}\left\{\frac{{\left\|\nabla f_j({\overline{x}}^{(k)})-\nabla f_i({\overline{x}}^{(k)})\right\|}_{2}}{n_k}\right\},
\end{align}
as $w\rightarrow0$. This implies that for sufficiently small $w$ the 
interval $[\gamma_{\min},\gamma_{\max})$ becomes wider, so that we can find a value of $\gamma$ in the range in an easier manner. 
This example indicates that if $(w_{\{i,j\}})_{\{i,j\}\in\mathcal{E}}$ are given so that they reflect the true cluster structure sufficiently, NL returns the true clusters with some $\gamma\in[\gamma_{\min},\gamma_{\max})$, namely, at some point on the cluster path, as demonstrated in Figure \ref{visual1}.
\end{remark}

\begin{remark}
Since $f_i(x_i)=\frac{1}{2}(b_i-a_i^\top x_i)^2$ is not strictly convex for $p\ge2$, Theorem \ref{thm:clusterrecover} does not apply to optimization problem \eqref{eq:nw_regress}. For example, it can be applied if $\frac{\varepsilon}{2}\|x_i\|_2^2$ is added to $f_i(x_i)$ where $\varepsilon>0$.
\end{remark}

\begin{remark}
While our result covers the case where $f_i(x_i)=\frac{1}{2}\|x_i-a_i\|_2^2$ for all $i\in\mathcal{V}$, i.e., Convex Clustering, Theorem \ref{thm:clusterrecover} is slightly weaker than the result of \citet{sun2021convex} for Convex Clustering because of the generalization beyond Convex Clustering. In their result, the thresholds corresponding to $\gamma_{\max}$ and $\gamma_{\min}$, between which recovery of true clusters is guaranteed, are given, respectively, by
\begin{align}
\gamma_{\max}'&=\min_{k\neq k'}\left\{\frac{{\|a^{(k)}-a^{(k')}\|}_2}{\frac{1}{n_k}\sum\limits_{l\neq k}w^{(k,l)}+\frac{1}{n_{k'}}\sum\limits_{l\neq k'}w^{(k',l)}}\right\},\\
\gamma_{\min}'&=\max_{k}\max_{\substack{i,j\in {C}_{k}\\ i\neq j}}\left\{\frac{{\left\|a_i-a_j\right\|}_{2}}{n_kw_{\{i,j\}}-\sum_{l\neq k}\left|w_i^{(l)}-w_j^{(l)}\right|}\right\},
\end{align}
where $a^{(k)}=\frac{1}{n_k}\sum_{i\in C_k}a_i$.
Applying our result and from $\sum_{l\neq k}\left|w_i^{(l)}-w_j^{(l)}\right|\le\mu_{ij}^{(k)}$, we have
\begin{align}
\gamma_{\max}&=\min_{k\neq k'}\left\{\frac{{\|a^{(k)}-a^{(k')}\|}_2}{\frac{1}{n_k}\sum\limits_{l\neq k}w^{(k,l)}+\frac{1}{n_{k'}}\sum\limits_{l\neq k'}w^{(k',l)}}\right\}=\gamma_{\max}',\\
\gamma_{\min}&=\max_{k}\max_{\substack{i,j\in {C}_{k}\\ i\neq j}}\left\{\frac{{\left\|a_i-a_j\right\|}_{2}}{n_kw_{\{i,j\}}-\mu_{ij}^{(k)}}\right\}\ge\gamma_{\min}'.
\end{align}
This shows their result admits a wider interval $[\gamma_{\min}',\gamma_{\max}')\supset[\gamma_{\min},\gamma_{\max})$.
\end{remark}

\begin{proof}[Proof of Statement 1]
Let $(x^{(1)*},\ldots,x^{(N)*})$ be an optimal solution of the following problem,
\begin{align}
\underset{x^{(1)},\ldots,x^{(N)}}{\text{minimize}} &\quad \sum_{k=1}^N f^{(k)}(x^{(k)})+\gamma\sum_{k<l}w^{(k,l)}{\|x^{(k)}-x^{(l)}\|}_2,
\label{problem:centroid}
\end{align}
which is equivalent to NL \eqref{convexclustering} with the symbols introduced in the statement of the theorem. 

We first show that $\gamma<\gamma_{\max}$ implies $x^{(k)*}\neq x^{(k')*}$ for all $k\neq k'$.
The optimality condition of \eqref{problem:centroid} is then given by
\begin{align}
\nabla f^{(k)}(x^{(k)*})+\gamma\sum_{l\neq k}w^{(k,l)}z^{(k,l)}=0,\qquad k\in[N],
\label{eq:FOC:centroid}
\end{align}
where $z^{(k,k')}\in\partial{\|x^{(k)*}-x^{(k')*}\|}_2$ and $z^{(k,k')}=-z^{(k',k)}$, for any $k, k'\in[N]$ such that $k\neq k'$.
Here, $\partial\|x^{(k)*}-x^{(k')*}\|_2$ denotes the subdifferential of $\|\cdot\|_2$ at $x^{(k)*}-x^{(k')*}$, and the subdifferential of $\|\cdot\|_2$ at $x$ is given by
\begin{align}
\partial\|x\|_2=
\left\{
\begin{array}{ll}
\{\frac{x}{\|x\|_2}\}, &x\neq0, \\
\{z\in\mathbb{R}^p\mid\|z\|_2\le1\}, &x=0.
\end{array}
\right.
\end{align}
By noting that $\|z^{(k,k')}\|_2\le1$, and combining it with the triangle inequality and the equation \eqref{eq:FOC:centroid}, we obtain
\begin{align}
\begin{split}
\|\nabla f^{(k)}(x^{(k)*})\|_2 &\leq \gamma\sum_{l\neq k}w^{(k,l)}\|z^{(k,l)}\|_2 \\
                               &\leq \gamma\sum_{l\neq k}w^{(k,l)}
\label{gradient_bound}
\end{split}
\end{align}
for all $k\in[N]$.
Since $f^{(k)}$ is $\alpha_k$-strongly convex, we have for arbitrary $k, k'\in[N]$ such that $k\neq k'$, 
\begin{align}
\begin{split}
{\|{\overline{x}}^{(k)}-{\overline{x}}^{(k')}\|}_2 
 &\leq \|{\overline{x}}^{(k)}-x^{(k)*}\|_2+\|x^{(k)*}-x^{(k')*}\|_2+\|x^{(k')*}-{\overline{x}}^{(k')}\|_2 \\
 &\leq \|x^{(k)*}-x^{(k')*}\|_2+\frac{1}{\alpha_k}\|\nabla f^{(k)}(x^{(k)*})\|_2+\frac{1}{\alpha_{k'}}\|\nabla f^{(k')}(x^{(k')*})\|_2 \\
 &\leq \|x^{(k)*}-x^{(k')*}\|_2+\gamma\Big(\frac{1}{\alpha_k}\sum_{l\neq k}w^{(k,l)}+\frac{1}{\alpha_{k'}}\sum_{l\neq k'}w^{(k',l)}\Big),
\label{error_bound}
\end{split}
\end{align}
where the first inequality is due to the triangle inequality, the second one follows from \eqref{l_t_bound}, and the final one from \eqref{gradient_bound}.
If the term to the right of $\gamma$ on the right-hand side of \eqref{error_bound} is equal to zero, then $\|x^{(k)*}-x^{(k')*}\|_2>0$ holds by the assumption that ${\overline{x}}^{(k)}\neq {\overline{x}}^{(k')}$ for $k\neq k'$.
Otherwise, combining the inequality \eqref{error_bound} and the definition of $\gamma_{\max}$, we obtain
\begin{align*}
\|x^{(k)*}-x^{(k')*}\|_2
 & \geq \|{\overline{x}}^{(k)}-{\overline{x}}^{(k')}\|_2-\gamma\Big(\frac{1}{\alpha_k}\sum_{l\neq k}w^{(k,l)}+\frac{1}{\alpha_{k'}}\sum_{l\neq k'}w^{(k',l)}\Big)\\
 & = \left(\frac{\|{\overline{x}}^{(k)}-{\overline{x}}^{(k')}\|_2}{\frac{1}{\alpha_k}\sum_{l\neq k}w^{(k,l)}+\frac{1}{\alpha_{k'}}\sum_{l\neq k'}w^{(k',l)}}-\gamma\right)\Big(\frac{1}{\alpha_k}\sum_{l\neq k}w^{(k,l)}+\frac{1}{\alpha_{k'}}\sum_{l\neq k'}w^{(k',l)}\Big)\\
 & \geq \left(\gamma_{\max}-\gamma\right)\Big(\frac{1}{\alpha_k}\sum_{l\neq k}w^{(k,l)}+\frac{1}{\alpha_{k'}}\sum_{l\neq k'}w^{(k',l)}\Big)\\
 & >0.
\end{align*}
Thus, $x^{(k)*}\neq x^{(k')*}$ for all $k\neq k'$. 

Next, we show that $\gamma_{\min}\leq\gamma$ implies $x_i^*= x^{(k)*}$ for all $i\in C_k$, $k\in[N]$. 
To this end, we now prove that the optimality condition of \eqref{nlasso} is satisfied, that is, there exists $(z_{ij})_{i\neq j}$ such that $z_{ij}\in\partial\|x_i^*-x_j^*\|_2$ and $z_{ij}=-z_{ji}$ for all $i\neq j,\ i, j\in\mathcal{V}$, and
\begin{align}
\nabla f_i(x_i^*)+\gamma\sum_{j\neq i}w_{\{i,j\}}z_{ij}=0
\end{align}
for all $i\in\mathcal{V}$.
Let
\begin{align*}
z_{ij}^*:=
\left\{
\begin{array}{ll}
z^{(k,k')}, &(i\in {C}_{k},\ j\in {C}_{k'},\ k\neq k'), \\
\frac{1}{n_kw_{\{i,j\}}}\left\{\frac{1}{\gamma}\left(\nabla f_j(x^{(k)*})-\nabla f_i(x^{(k)*})\right)+p_j^{(k)}-p_i^{(k)}\right\}, &(i,j\in {C}_{k},\ i\neq j),
\end{array}
\right.
\end{align*}
where
\begin{align*}
p_i^{(k)}:=\sum_{l\neq k}\left(w_i^{(l)}-\frac{1}{n_k}w^{(k,l)}\right)z^{(k,l)}.
\end{align*}
Obviously, it holds that $z_{ij}^*=-z_{ji}^*$ for any $i\neq j,\ i, j\in\mathcal{V}$. For all $i\in {C}_{k},\ j\in {C}_{k'},\ k\neq k'$, it is valid $z_{ij}^*=z^{(k,k')}\in\partial{\|x^{(k)*}-x^{(k')*}\|}_2=\partial{\|x_i^*-x_j^*\|}_2$. 
For arbitrary $i,j\in {C}_{k}$, $k\in[N]$, we have
\begin{align*}
\lefteqn{\|\nabla f_j(x^{(k)*})-\nabla f_i(x^{(k)*})\|_2}\\
 &\leq \|\nabla f_j(x^{(k)*})-\nabla f_j({\overline{x}}^{(k)})\|_2+\|\nabla f_j({\overline{x}}^{(k)})-\nabla f_i({\overline{x}}^{(k)})\|_2+\|\nabla f_i({\overline{x}}^{(k)})-\nabla f_i(x^{(k)*})\|_2\\
 &\leq \|\nabla f_j({\overline{x}}^{(k)})-\nabla f_i({\overline{x}}^{(k)})\|_2+(L_i+L_j)\|{\overline{x}}^{(k)}-x^{(k)*}\|_2\\
 &\leq \|\nabla f_j({\overline{x}}^{(k)})-\nabla f_i({\overline{x}}^{(k)})\|_2+\frac{\gamma(L_i+L_j)}{\alpha_k}\sum_{l\neq k}w^{(k,l)},
\end{align*}
where the first inequality follows from the triangle inequality, the second one from the $L_i$-smoothness of $f_i$, and the third one from \eqref{l_t_bound} and \eqref{gradient_bound}. 
Accordingly, we obtain
\begin{align*}
   &{\|z_{ij}^*\|}_{2} \\
   &= \frac{1}{n_kw_{\{i,j\}}}{\left\|\frac{1}{\gamma}\left(\nabla f_j(x^{(k)*})-\nabla f_i(x^{(k)*})\right)+p_j^{(k)}-p_i^{(k)}\right\|}_{2}\\
   &\leq \frac{1}{n_kw_{\{i,j\}}\gamma}{\left\|\nabla f_j(x^{(k)*})-\nabla f_i(x^{(k)*})\right\|}_{2}+\frac{1}{n_kw_{\{i,j\}}}{\left\|p_j^{(k)}-p_i^{(k)}\right\|}_{2}\\
   &\leq \frac{1}{n_kw_{\{i,j\}}\gamma}{\left\|\nabla f_j({\overline{x}}^{(k)})-\nabla f_i({\overline{x}}^{(k)})\right\|}_{2}+\frac{1}{n_kw_{\{i,j\}}}\Big(\sum_{l\neq k}\left|w_i^{(l)}-w_j^{(l)}\right|+\frac{L_i+L_j}{\alpha_k}\sum_{l\neq k}w^{(k,l)}\Big)\\
   &\leq \frac{1}{n_kw_{\{i,j\}}\gamma_{\min}}{\left\|\nabla f_j({\overline{x}}^{(k)})-\nabla f_i({\overline{x}}^{(k)})\right\|}_{2}+\frac{\mu_{ij}^{(k)}}{n_kw_{\{i,j\}}}\\
   &\leq \frac{n_kw_{\{i,j\}}-\mu_{ij}^{(k)}}{n_kw_{\{i,j\}}}+\frac{\mu_{ij}^{(k)}}{n_kw_{\{i,j\}}}\\
   &= 1,
\end{align*}
where the first inequality follows from the triangle inequality, the second one from the definition of $p^{(k)}:=(p_i^{(k)})_{i\in C_k}$ and the previous inequality, the third and fourth ones from the definitions of $\mu_{ij}^{(k)}$ and $\gamma_{\min}$, respectively. 
This implies $z_{ij}^*\in\partial{\|0\|}_2=\partial{\|x_i^*-x_j^*\|}_2$ for all $i,j\in {C}_{k}$, $k\in[N]$. 
On the other hand, we have
\begin{align*}
\lefteqn{\nabla f_i(x_i^*)+\gamma\sum_{j\neq i}w_{\{i,j\}}z_{ij}^*}\\
   &= \nabla f_i(x^{(k)*})+\gamma\sum_{l\neq k}w_i^{(l)}z^{(k,l)}\\
   &\hspace{12em}+\gamma\sum_{\substack{j\neq i\\ j\in {C}_{k}}}w_{\{i,j\}}\frac{1}{n_kw_{\{i,j\}}}\left\{\frac{1}{\gamma}\left(\nabla f_j(x^{(k)*})-\nabla f_i(x^{(k)*})\right)+p_j^{(k)}-p_i^{(k)}\right\}\\
   &= \frac{1}{n_k}\sum_{j\in {C}_{k}}\nabla f_j(x^{(k)*})+\gamma\sum_{l\neq k}w_i^{(l)}z^{(k,l)}+\frac{1}{n_k}\gamma\sum_{\substack{j\neq i\\ j\in {C}_{k}}}\sum_{l\neq k}\left(w_j^{(l)}-w_i^{(l)}\right)z^{(k,l)}\\
   &= \frac{1}{n_k}\sum_{j\in {C}_{k}}\nabla f_j(x^{(k)*})+\frac{1}{n_k}\gamma\sum_{l\neq k}\sum_{j\in {C}_{k}}w_j^{(l)}z^{(k,l)}\\
   &= \frac{1}{n_k}\left(\nabla f^{(k)}(x^{(k)*})+\gamma\sum_{l\neq k}w^{(k,l)}z^{(k,l)}\right)\\
   &= 0,
\end{align*}
where the first equality is by the definition of $z^*_{ij}$, the second one is by the definitions of $p^{(k)}$ and $w^{(l)}$, and the final equality follows from \eqref{eq:FOC:centroid}. 
These results show that $(x^*_1,...,x^*_n)$ is the unique optimal solution of \eqref{nlasso} because of the strict convexity of the objective function of \eqref{nlasso}. 
Thus we conclude that $\overline{\mathcal{P}}$ perfectly recovers $\mathcal{P}$.
\vspace{.25em}

\noindent
{\it Proof of Statement 2.}~ 
Observe that if $x^{(1)*}=\cdots=x^{(N)*}$, then $x^{(k)*}=\overline{x}$ holds for $k\in[N]$ because of the definition of $\overline{x}$ and the strict convexity of $\sum\limits_{i\in\mathcal{V}}f_i(x)$.
From the inequality \eqref{gradient_bound}, we have
\begin{align*}
\max_{k}\frac{{\|\nabla f^{(k)}(\overline{x})\|}_2}{\sum_{l\neq k}w^{(k,l)}}\leq\gamma.
\end{align*}
Therefore, if $\gamma<\max_{k}\frac{{\|\nabla f^{(k)}(\overline{x})\|}_2}{\sum_{l\neq k}w^{(k,l)}}$, then ${\overline{x}}^{(1)}=\cdots={\overline{x}}^{(N)}$ does not hold. 
In addition, if we take $\gamma\geq\gamma_{\min}$, then $x_i^*={\overline{x}}^{(k)} (i\in C_k)$ is the optimal solution of \eqref{nlasso}, as in the proof of Statement 1. Thus $\overline{\mathcal{P}}$ is a non-trivial coarsening of $\mathcal{P}$.
 \end{proof}

\section{Network Trimmed Lasso}
\label{sec:NTL}
In the previous section, we see that when the prior information $(w_{\{i,j\}})_{\{i,j\}\in\mathcal{E}}$ is given adequately, we can use NL for clustering.
However, in the absence of the prior information, clustering by NL does not work well, as seen in Figure \ref{visual1}.
Rather than not forming 
reasonable 
clusters, NL does not even form non-trivial clusters. 
Furthermore, it might not be easy to adequately define prior information for other tasks such as regression, as will be demonstrated in Section \ref{sec:numerical}. 
In this section, we consider an extension of NL by directly incorporating a non-convex constraint to enhance the cluster structure. 

\subsection{Cardinality-constrained formulation and its equivalent continuous penalty reformulation}
In NL \eqref{nlasso}, the cluster structure is captured by the number of non-zero components of the vectors $(\|x_i-x_j\|_2)_{\{i,j\}\in\mathcal{E}}$.
To directly control the structure, we consider a minimization problem \eqref{obj:card}--\eqref{const:card}, where the fitting of the dataset to models is optimized under a designated cardinality of non-zero components of the vector:
\begin{align}
\underset{x_1,...,x_n}{\mbox{minimize}} & \quad \sum_{i\in\mathcal{V}}f_i(x_i) 
\label{obj:card} \\
\text{subject to}                       & \quad \Big|\big\{\{i,j\}\in\mathcal{E}:\|x_i-x_j\|_2>0\big\}\Big|\leq K,
\label{const:card}
\end{align}
where $K$ is a non-negative integer such that $K\leq|\mathcal{E}|$. 
As $K$ decreases, the nodes agglomerate and form clusters. 
Note that if $K\in\{1,...,|\mathcal{E}|-1\}$, \eqref{obj:card}--\eqref{const:card} is a non-trivial optimization problem. 
\citet{hocking2011clusterpath} treat Convex Clustering as a convex relaxation of problem \eqref{obj:card}--\eqref{const:card}. 

While it is easier to interpret the hyperparameter $K$ in \eqref{obj:card}--\eqref{const:card} than $\gamma$ in NL \eqref{nlasso}, the left-hand side of  \eqref{const:card} is a discontinuous function in $(x_1,...,x_n)$ and is known to be difficult to attain the global optimality of \eqref{obj:card}--\eqref{const:card} in general.
Therefore, we approach the problem by rewriting the cardinality constraint with an equivalent continuous counterpart. 

$\xi:=(\|x_i-x_j\|_2)_{\{i,j\}\in\mathcal{E}}\in\mathbb{R}^{|\mathcal{E}|}$, and let us denote the sum of the $|\mathcal{E}|-K$ smallest components of the vector $\xi$ by 
\begin{align}
\tau_K(x_1,...,x_n)&:=\xi_{(K+1)}+\cdots+\xi_{(|\mathcal{E}|)}\mbox{ with }\xi=(\|x_i-x_j\|_2)_{\{i,j\}\in\mathcal{E}},
\label{def:tau}
\end{align}
where $\xi_{(i)}$ denotes the $i$-th largest component of $\xi$. 
Note that $\tau_K(x_1,...,x_n)\geq 0$ for any $(x_1,...,x_n)$. 
It is easy to see that the problem \eqref{obj:card}--\eqref{const:card} is equivalent to the following problem:
\begin{align}
\underset{x_1,...,x_n}{\mbox{minimize}} & \quad \sum_{i\in\mathcal{V}}f_i(x_i) 
\label{obj:TK_form} \\
\text{subject to}                       & \quad \tau_K(x_1,...,x_n)=0,
\label{const:TK_form}
\end{align}
by noting that \eqref{const:card} and \eqref{const:TK_form} are equivalent (see \citet[Subsection 5.2]{gotoh2018dc}). 
Note that \eqref{obj:TK_form}--\eqref{const:TK_form} is a continuous optimization problem if $f_i$ are continuous, whereas \eqref{obj:card}--\eqref{const:card} is not the case in that the constraint \eqref{const:card} includes a discontinuous function on its left-hand side. Then we introduce the following penalty form:

\begin{align}
\underset{x_1,...,x_n}{\mbox{minimize}} & \quad \sum_{i\in\mathcal{V}}f_i(x_i)+\gamma\tau_K(x_1,...,x_n),
\label{dcnl}
\end{align}
where $\gamma>0$. 
The second term of the objective function of \eqref{dcnl} plays a role of a penalty function of the cardinality constraint \eqref{const:card} in that (i) $\tau_K(x_1,...,x_n)\geq 0$ for all $(x_1,...,x_n)$, and (ii) $\tau_K(x_1,...,x_n)>0$ if and only if $\big|\big\{\{i,j\}\in\mathcal{E}:\xi_{\{i,j\}}>0\big\}\big|> K$. 

We call the problem \eqref{dcnl} \emph{Network Trimmed Lasso} (NTL for short). 
If we set $K=0$, the problem \eqref{dcnl} is reduced to NL \eqref{nlasso} with $w_{\{i,j\}}=1$ for all $\{i,j\}\in\mathcal{E}$.

While \eqref{dcnl} is now an unconstrained problem, another parameter $\gamma$ is introduced instead. 
We will show below that if we take $\gamma$ large enough, \eqref{dcnl} is guaranteed to be equivalent to the constrained problem \eqref{obj:TK_form}--\eqref{const:TK_form}, and accordingly, to the cardinality-constrained problem \eqref{obj:card}--\eqref{const:card}. 

\begin{theorem}\label{e.p.p.}
\begin{list}{}{
\itemindent 0.5em
\leftmargin 0em
}
\item[1.] Suppose that $f_i$ is $L_i$-smooth for each $i\in\mathcal{V}$, and let $\bm{x}^\gamma:=(x^\gamma_1,...,x^\gamma_n)$ be an optimal solution of \eqref{dcnl}.
Suppose that there exists $C>0$ such that $\|x^\gamma_i\|_{2}\leq C$ for all $i\in\mathcal{V}$ and any $\gamma>0$. 
Then $\bm{x}^\gamma$ is optimal to \eqref{obj:TK_form}--\eqref{const:TK_form} if
\begin{align}
\gamma>\sum_{i\in\mathcal{V}}({\|\nabla f_i(0)\|}_{2}+2L_{i}C).
\label{e.p.p.inequality}
\end{align}
\item[2.] In addition to the $L_i$-smoothness, suppose that $f_i$ is convex for each $i\in\mathcal{V}$, and let $\bm{x}^\gamma:=(x^\gamma_1,...,x^\gamma_n)$ be a locally optimal solution of \eqref{dcnl}.
Suppose that there exists $C>0$ such that $\|x^\gamma_i\|_{2}\leq C$ for all $i\in\mathcal{V}$ and any $\gamma>0$.
Then $\bm{x}^\gamma$ is locally optimal to \eqref{obj:TK_form}--\eqref{const:TK_form} if the inequality \eqref{e.p.p.inequality} holds.
\end{list}
\end{theorem}

\begin{proof}[Proof of Statement 1]
Note that if $\tau_K(x^\gamma_1,...,x^\gamma_n)=0$ holds, $\bm{x}^\gamma$ is a minimizer of \eqref{obj:TK_form}--\eqref{const:TK_form}. 
Assume that $\tau_K(x^\gamma_1,...,x^\gamma_n)>0$. In this case, let $\mathcal{E}'\subset\mathcal{E}$ be a set of edges $\{i,j\}\in\mathcal{E}$ whose $\|x^\gamma_i-x^\gamma_j\|_2$ is in the smallest $|\mathcal{E}|-K$ components and divide $\mathcal{V}$ into connected components $\mathcal{\mathcal{C}}_{1},\ldots,\mathcal{C}_{m}$ of the graph $\left(\mathcal{V},\mathcal{E}'\right)$, then we set
\begin{align*}
x'_{i}&:=\sum_{j\in \mathcal{C}_{k}}\frac{{x}_{j}^{\gamma}}{|\mathcal{C}_{k}|},
\end{align*}
for $i\in\mathcal{C}_{k}$, $k\in[m]$. 
Obviously, $\tau_K(x'_1,...,x'_n)=0$ and $\|x'_i\|_{2}\leq C$ are fulfilled. 
If $i,j\in \mathcal{C}_{k}, k\in[m]$ and $i\neq j$, then there exists a simple path between $i$ and $j$ on $(\mathcal{V},\mathcal{E}')$, so
\begin{align*}
\|x_{i}^{\gamma}-x_{j}^{\gamma}\|_{2}
 &\leq \sum_{\{i',j'\}\in\mathcal{E}'}\|{x}_{i'}^{\gamma}-{x}_{j'}^{\gamma}\|_2\\
 &\leq \tau_K(x^\gamma_1,...,x^\gamma_n).
\end{align*}
Thus we obtain
\begin{align}
\begin{split}
\|x'_{i}-{x}_{i}^{\gamma}\|_{2}
  &\leq \sum_{j\in \mathcal{C}_{k}}\frac{{\|{x}_{i}^{\gamma}-{x}_{j}^{\gamma}\|}_{2}}{|\mathcal{C}_{k}|} \\
  &\leq \sum_{j\in \mathcal{C}_{k}}\frac{\tau_K(x^\gamma_1,...,x^\gamma_n)}{|\mathcal{C}_{k}|} \\
  &\leq \tau_K(x^\gamma_1,...,x^\gamma_n),
\label{diffofsuggest}
\end{split}
\end{align}
for all $i\in \mathcal{C}_{k}, k\in[m]$. 
 From ${\|{x}_{i}^{\gamma}\|}_{2}\leq C$ and $f_i$'s $L_i$-smoothness, we have
\begin{align}
\begin{split}
\lefteqn{\sum_{i\in\mathcal{V}}f(x_i^{\gamma})+\gamma\tau_K(\bm{x}^{\gamma})-\left(\sum_{i\in\mathcal{V}}f(x'_i)+\gamma\tau_K(\bm{x}')\right)} \\
    &\geq \gamma\tau_K(\bm{x}^{\gamma})+\sum_{i\in\mathcal{V}}\left({\nabla f_i(x_i^{\gamma})}^\top(x_i^{\gamma}-x'_i)-\frac{L_i}{2}\|x_i^{\gamma}-x'_i\|_2^2\right) \\
    &\geq \gamma\tau_K(\bm{x}^{\gamma})-\sum_{i\in\mathcal{V}}{\|x_i^{\gamma}-x'_i\|}_{2}\left({\|\nabla f_i(x_i^{\gamma})\|}_{2}+\frac{L_i}{2}{\|x_i^{\gamma}-x'_i\|}_{2}\right) \\
    &\geq \gamma\tau_K(\bm{x}^{\gamma})-\sum_{i\in\mathcal{V}}{\|x_i^{\gamma}-x'_i\|}_{2}\left({\|\nabla f_i(0)\|}_{2}+{\|\nabla f_i(x_i^{\gamma})-\nabla f_i(0)\|}_{2}+\frac{L_i}{2}\left({\|x_i^{\gamma}\|}_{2}+{\|x'_i\|}_{2}\right)\right) \\
    &\geq \gamma\tau_K(\bm{x}^{\gamma})-\sum_{i\in\mathcal{V}}{\|x_i^{\gamma}-x'_i\|}_{2}\left({\|\nabla f_i(0)\|}_{2}+L_i{\|x_i^{\gamma}\|}_{2}+\frac{L_i}{2}\left({\|x_i^{\gamma}\|}_{2}+{\|x'_i\|}_{2}\right)\right) \\
    &\geq \gamma\tau_K(\bm{x}^{\gamma})-\sum_{i\in\mathcal{V}}{\|x_i^{\gamma}-x'_i\|}_{2}\left({\|\nabla f_i(0)\|}_{2}+2L_i C\right) \\
    &\geq \gamma\tau_K(\bm{x}^{\gamma})-\sum_{i\in\mathcal{V}}\tau_K(\bm{x}^{\gamma})\left({\|\nabla f_i(0)\|}_{2}+2L_i C\right) \\
    &= \tau_K(\bm{x}^{\gamma})\left(\gamma-\sum_{i\in\mathcal{V}}\left({\|\nabla f_i(0)\|}_{2}+2L_i C\right)\right) \\
    &> 0,
\label{suggestineq}
\end{split}
\end{align}
where the first and fourth inequalities follow from the $L_i$-smoothness of $f_i$, where we apply the inequality \eqref{inv_descent_lemma} to the first one, the second one from the Cauchy-Schwarz inequality, the third one from the triangle inequality, the fifth one from the boundedness of $\bm{x}^{\gamma}$ and $\bm{x}'$, the sixth one from the inequality \eqref{diffofsuggest}. The above inequality \eqref{suggestineq} contradicts the optimality of $\bm{x}^{\gamma}$.

\noindent
{\it Proof of Statement 2.}~ 
Note that if $\tau_K(x^\gamma_1,...,x^\gamma_n)=0$ is fulfilled, $\bm{x}^\gamma$ is a local minimizer of \eqref{obj:TK_form}--\eqref{const:TK_form}. Assume $\tau_K(x^\gamma_1,...,x^\gamma_n)>0$. Let us define $\mathcal{E}'$ as in the proof of the statement 1., let
\begin{align*}
v_{\{i,j\}}:=
\left\{
\begin{array}{ll}
1, &\{i,j\}\in\mathcal{E}', \\
0, &otherwise,
\end{array}
\right.
\end{align*}
and consider the following problem:
\begin{align}
\underset{x_1,...,x_n}{\mbox{minimize}} & \quad \sum_{i\in\mathcal{V}}f_i(x_i)+\gamma\sum_{\{i,j\}\in\mathcal{E}}v_{\{i,j\}}\|x_i-x_j\|_{2}.
\label{scrdcnl}
\end{align}
Note that $\ \sum_{\{i,j\}\in\mathcal{E}}v_{\{i,j\}}\|x^\gamma_i-x^\gamma_j\|_{2}=\tau_K(x^\gamma_1,...,x^\gamma_n)$. We have $\sum_{\{i,j\}\in\mathcal{E}}v_{\{i,j\}}\|x_i-x_j\|_{2}\geq\tau_K(x_1,...,x_n)$ for any $(x_1,...,x_n)$ by the definition of $\tau_K$. Since $\bm{x}^\gamma$ is locally optimal to \eqref{dcnl}, $\bm{x}^\gamma$ is a local minimizer of \eqref{scrdcnl}. Because of the convexity of \eqref{scrdcnl}, $\bm{x}^\gamma$ is optimal to \eqref{scrdcnl}. Determining $\bm{x}'$ in the same way as in the proof for the statement 1., we have $\sum_{\{i,j\}\in\mathcal{E}}v_{\{i,j\}}\|x'_i-x'_j\|_{2}=0$, $\|x'_i\|_{2}\leq C$, and the inequality \eqref{diffofsuggest}. By the same calculation as in \eqref{suggestineq}, we reach the contradiction to the fact that $\bm{x}^\gamma$ is optimal to \eqref{scrdcnl}.
\end{proof}

By Statement 1. of Theorem \ref{e.p.p.}, we are motivated to solve NTL~\eqref{dcnl} instead of the cardinality-constrained problem \eqref{obj:card}--\eqref{const:card} since NTL~\eqref{dcnl} is an unconstrained  minimization of a continuous function while \eqref{obj:card}--\eqref{const:card} involves a constraint defined by a discontinuous function. 
Despite the continuity of the objective function, developing a global optimization algorithm for \eqref{dcnl} is not easy especially when the number of variables is large. 
On the other hand, Statement 2. of Theorem \ref{e.p.p.} yields conditions under which a locally optimal solution to \eqref{obj:card}--\eqref{const:card} is obtained by a locally optimal solution to NTL~\eqref{dcnl}, which is attainable by, for example, Proximal ADMM (\citet{li2015global}) as shown in the next section.   

Both statements of Theorem \ref{e.p.p.} suppose that the size of solution set is bounded by a constant $C$. 
In the following, we will see a few examples where values of $C$ can be explicitly given.  

\begin{example}[Network Trimmed Lasso for ordinary clustering]\label{dcc}
Consider the clustering problem of a data set $a_i\in\mathbb{R}^p,i\in\mathcal{V}$ with $f_i(x_i)=\frac{1}{2}\|x_i-a_i\|_2^2,i\in\mathcal{V}$. 
(Note that we do not limit to the case where $\mathcal{E}=\left\{\{i,j\}\mid i\neq j, i,j\in\mathcal{V}\right\}$.) 
The NTL then becomes 
\begin{align}
\underset{x_1,...,x_n\in\mathbb{R}^p}{\mbox{minimize}} & \quad \frac{1}{2}\sum_{i\in\mathcal{V}} \|x_i-a_i\|_2^2+\gamma\tau_K(x_1,...,x_n).
\label{dcclustering}
\end{align}
\end{example}
For this clustering problem, we can find a threshold value of $\gamma$ of Theorem \ref{e.p.p.} explicitly in a simple manner. 
To see this, first observe the following lemma, which shows the boundedness of locally optimal solution to \eqref{dcclustering}. 
\begin{lemma}
\label{dccbound}
Let $C=\max_{i\in\mathcal{V}}{\|a_i\|}_2$. 
 For any $\gamma>0$, any locally optimal solution $\bm{x}^*$ of \eqref{dcclustering} satisfies ${\|x_i^*\|}_2\leq C$ for all $i\in\mathcal{V}$.
\end{lemma}

\begin{proof}
Let $\mathcal{E}'\subset\mathcal{E}$ be a set of edges $\{i,j\}\in\mathcal{E}$ whose $\|x^*_i-x^*_j\|_2$ is in the smallest $|\mathcal{E}|-K$ components out of all the $|\mathcal{E}|$ components, then we define
\begin{align*}
v_{\{i,j\}}:=
\left\{
\begin{array}{ll}
1, &\{i,j\}\in\mathcal{E}', \\
0, &otherwise,
\end{array}
\right.
\end{align*}
and consider the following problem:
\begin{align}
\underset{x_1,...,x_n}{\mbox{minimize}} & \quad \frac{1}{2}\sum_{i\in\mathcal{V}} \|x_i-a_i\|_2^2+\gamma\sum_{\{i,j\}\in\mathcal{E}}v_{\{i,j\}}\|x_i-x_j\|_{2}.
\label{scrdcc}
\end{align}
From the convexity of $\frac{1}{2}\sum_{i\in\mathcal{V}} \|x_i-a_i\|_2^2$, as in the proof of the second statement of Theorem \ref{e.p.p.}, $\bm{x}^*$ is an optimal solution of \eqref{scrdcc}.
Assume that there exists an $i\in\mathcal{V}$ such that $\|x_i^*\|_2>C$.
Let $O=\{i\mid{\|x^*_i\|}_{2}>C\}$, and define
\begin{align*}
x'_i:=\left\{
\begin{array}{cl}
\frac{R}{\|x^*_i\|_2}x^*_i, & i\in O, \\
x^*_i, & i\notin O.
\end{array}
\right.
\end{align*}
Obviously, it is valid that
\begin{align*}
\|x^*_i-a_i\|_2^2&=\|x'_i-a_i\|_2^2,
\end{align*}
for $i\notin O$, and
\begin{equation*}
\|x^*_i-x^*_j\|_2=\|x'_i-x'_j\|_{2},
\end{equation*}
for $i,j\notin O$. 
Because $x'_i$ is the projection of $x^*_i$ onto the closed convex set $\{x\in\mathbb{R}^{p}:\|x\|_2\leq R\}$, we obtain
\begin{align*}
\|x^*_i-x^*_j\|_2^2
    &= \|x^*_i-x'_j\|_2^2 \\
    &= \|x^*_i-x'_i+x'_i-x'_j\|_2^2 \\
    &= {\|x^*_i-x'_i\|}_2^2+2(x^*_i-x'_i)^\top(x'_i-x'_j)+\|x'_i-x'_j\|_2^2 \\
    &\geq \big\|\left(1-\frac{R}{\|x^*_i\|_2}\right)x^*_i\big\|_2^2+\|x'_i-x'_j\|_2^2 \\
    &> \|x'_i-x'_j\|_2^2,
\end{align*}
for $i\in O, j\notin O$. In the same way, we have
\begin{align*}
 \|x^*_i-x^*_j\|_2^2 &= \|x^*_i-x'_i+x'_i-x'_j+x'_j-x^*_j\|_2^2 \\
 &= \|x^*_i-x'_i+x'_j-x^*_j\|_2^2+2(x^*_i-x'_i)^\top(x'_i-x'_j)+2(x'_j-x^*_j)^\top(x'_i-x'_j)+\|x'_i-x'_j\|_2^2 \\
 &\geq \|x'_i-x'_j\|_2^2,
\end{align*}
for $i, j\in O$, and
\begin{align*}
\|x^*_i-a_i\|_2^2
  &= \|x^*_i-x'_i+x'_i-a_i\|_2^2 \\
  &= {\|x^*_i-x'_i\|}_2^2+2(x^*_i-x'_i)^\top(x'_i-a_i)+\|x'_i-a_i\|_2^2 \\
  &\geq \left\|\left(1-\frac{R}{\|x^*_i\|_2}\right)x^*_i\right\|_2^2+\|x'_i-a_i\|_2^2 \\
  &\geq \left(\|x^*_i\|_2-R\right)^2+\|x'_i-a_i\|_2^2 \\
  &> \|x'_i-a_i\|_2^2,
\end{align*}
for $i\in O$. 
This implies that
\begin{align}
\frac{1}{2}\sum_{i\in\mathcal{V}}\|x_i^*-a_i\|_2^2&>\frac{1}{2}\sum_{i\in\mathcal{V}}\|x'_i-a_i\|_2^2,\\
\sum_{\{i,j\}\in\mathcal{E}}v_{\{i,j\}}\|x^*_i-x^*_j\|_{2}&\ge\sum_{\{i,j\}\in\mathcal{E}}v_{\{i,j\}}\|x'_i-x'_j\|_{2}.
\end{align}
Thus we have
\begin{align}
\frac{1}{2}\sum_{i\in\mathcal{V}}\|x_i^*-a_i\|_2^2+\gamma\sum_{\{i,j\}\in\mathcal{E}}v_{\{i,j\}}\|x^*_i-x^*_j\|_{2}>\frac{1}{2}\sum_{i\in\mathcal{V}}\|x'_i-a_i\|_2^2+\gamma\sum_{\{i,j\}\in\mathcal{E}}v_{\{i,j\}}\|x'_i-x'_j\|_{2},
\end{align}
which contradicts the fact that $\bm{x}^*$ is optimal to \eqref{scrdcc}.
Consequently, we have $\|x_i^*\|_2\leq C$ for all $i\in\mathcal{V}$.
\end{proof}

 From Theorem~\ref{e.p.p.} and Lemma~\ref{dccbound}, we obtain the following result, which dictates an explicit threshold value of the penalty parameter $\gamma$ for ordinary clustering.
\begin{corollary}
\label{c.e.p.p.}
Let $C=\max_{i\in\mathcal{V}}\|a_i\|_2$. 
If $\gamma>3nC$, then any optimal solution (resp. locally optimal solution) of \eqref{dcclustering} is also optimal (resp. locally optimal) to the cardinality-constrained clustering problem (i.e., Problem \eqref{obj:card}--\eqref{const:card} with $f_i(x_i)=\frac{1}{2}\|x_i-a_i\|_2^2$).
\end{corollary}

\begin{proof}
Since $f_i(x_i)=\frac{1}{2}{\|x_i-a_i\|}_2^2$ is $1$-smooth and $\|\nabla f_i(0)\|_2=\|-a_i\|_2\leq C$, we have
\begin{align}
\sum_{i\in\mathcal{V}}\left({\|\nabla f_i(0)\|}_{2}+2L_{i}C\right)\leq\sum_{i\in\mathcal{V}}\left(C+2C\right)=3nC.
\end{align}
This completes the proof.
\end{proof}

Beyond the ordinary clustering problem, we can raises further examples where the threshold of $\gamma$ is derived.  
Consider a general case where $f_i$ is $\alpha_i$-strongly convex for all $i\in\mathcal{V}$. 
The following lemma claims that any locally optimal solution to NTL~\eqref{dcnl} is then bounded. 

\begin{lemma}
\label{dcnlbound}
Assume that $f_i$ is $\alpha_i$-strongly convex for all $i\in\mathcal{V}$.
Denote an unique optimizer of $\min f_i(x)$ by $\overline{x}_i$.
Let $C=\left(\frac{2}{\alpha}\sum_{j\in\mathcal{V}}\left(f_j(0)-f_j(\overline{x}_j)\right)\right)^\frac{1}{2}+\max_{i\in\mathcal{V}}\|\overline{x}_i\|_2$, where $\alpha=\min_{i\in\mathcal{V}}\alpha_i$.
Then for any $\gamma>0$, any locally optimal solution $\bm{x}^*$ of \eqref{dcnl} satisfies ${\|x_i^*\|}_2\leq C$ for all $i\in\mathcal{V}$.
\end{lemma}

\begin{proof}
From the convexity of $\sum_{i\in\mathcal{V}} f_i(x_i)$, as in the proof of the second statement of Theorem \ref{e.p.p.}, $\bm{x}^*$ is optimal to
\begin{align}
\underset{x_1,...,x_n}{\mbox{minimize}} & \quad \sum_{i\in\mathcal{V}} f_i(x_i)+\gamma\sum_{\{i,j\}\in\mathcal{E}}v_{\{i,j\}}\|x_i-x_j\|_{2},
\label{s_scrdcnl}
\end{align}
where $v_{\{i,j\}}$ is defined in the same way. Since $\bm{x}^*$ is optimal to \eqref{s_scrdcnl}, we have
\begin{align}
\sum_{i\in\mathcal{V}} f_i(x^*_i)\le\sum_{i\in\mathcal{V}} f_i(x_i)+\gamma\sum_{\{i,j\}\in\mathcal{E}}v_{\{i,j\}}\|x_i-x_j\|_{2}\le\sum_{i\in\mathcal{V}} f_i(0).
\label{by_opt_s_scrdcnl}
\end{align}
From the strong convexity of $f_i$, combining \eqref{by_opt_s_scrdcnl} and \eqref{plain_bound} yields
\begin{align}
\frac{\alpha}{2}\|x^*_i-\overline{x}_i\|_2^2 &\le\sum_{j\in\mathcal{V}}\frac{\alpha_j}{2}\|x^*_j-\overline{x}_j\|_2^2\\
&\le\sum_{j\in\mathcal{V}}\left(f_j(x^*_j)-f_j(\overline{x}_j)\right)\\
&\le\sum_{j\in\mathcal{V}}\left(f_j(0)-f_j(\overline{x}_j)\right)
\end{align}
for all $i\in\mathcal{V}$. Applying the triangle inequality to this, we get
\begin{align}
\|x^*_i\|_2 &\le\|x^*_i-\overline{x}_i\|_2+\|\overline{x}_i\|_2\\
&\le\left(\frac{2}{\alpha}\sum_{j\in\mathcal{V}}\left(f_j(0)-f_j(\overline{x}_j)\right)\right)^\frac{1}{2}+\|\overline{x}_i\|_2\\
&\le C.
\end{align}
This completes the proof.
\end{proof}

In the case where $f_i$ is the quadratic function $\frac{1}{2}x_i^\top A_ix_i-B_i^\top x_i$ with a positive definite matrix $A_i$, by using Lemma \ref{dcnlbound} a threshold value of $\gamma$ can be specified as follows.

\begin{corollary}
\label{q.e.p.p.}
Suppose that for all $i\in\mathcal{V}$, the matrix $A_i$ is positive definite and $f_i(x_i)=\frac{1}{2}x_i^\top A_ix_i-B_i^\top x_i$.
Let $C=\left(\frac{1}{\alpha}\sum_{i\in\mathcal{V}}B_i^\top A_i^{-1}B_i\right)^\frac{1}{2}+\max_{i\in\mathcal{V}}\|A_i^{-1}B_i\|_2$, where $\alpha=\min_{i\in\mathcal{V}}\lambda_{\min}(A_i)$.
If $\gamma>\sum_{i\in\mathcal{V}}({\|B_i\|}_{2}+2\lambda_{\max}(A_i)C)$, then any optimal solution (resp. locally optimal solution) of \eqref{dcnl} is also optimal (resp. locally optimal) to \eqref{obj:card}--\eqref{const:card}.
\end{corollary}

\begin{proof}
Note that for any $i\in\mathcal{V}$, $f_i$ is $\lambda_{\min}(A_i)$-strongly convex and $\lambda_{\max}(A_i)$-smooth, and the gradient and minimizer of $f_i$ are given by $\nabla f_i(x_i)=A_ix_i-B_i$ and $A_i^{-1}B_i$, respectively. By applying Theorem \ref{e.p.p.} and Lemma \ref{dcnlbound}, we have the desired result.
\end{proof}

\section{Algorithm}
\label{sec:algorithm}
In this section, we develop two algorithms to approach a solution of NTL \eqref{dcnl} and generate a cluster path with respect to the cardinality parameter $K$.

\subsection{ADMM}
As the first algorithm, we consider Alternating Direction Method of Multipliers (ADMM) (e.g., \citet{boyd2011distributed}). 
For NL (including Convex Clustering), \citet{chi2015splitting} and \citet{hallac2015network} propose a method based on ADMM.

In this subsection, we deal with a more general problem, which includes NTL~\eqref{dcnl} as a special case. 
Similar to the trimmed Lasso function \eqref{def:tau}, let us define the function $T_K$ on $\mathbb{R}^{pm}$ by
\begin{align}
T_K((z_k)_{k\in[m]})=\|z_{(K+1)}\|_2+\cdots+\|z_{(m)}\|_2,
\end{align}
where $K\in\{0,1,\ldots,m\}$, $z_k\in\mathbb{R}^p$, and $\|z_{(k)}\|_2$ denotes the $k$-th largest component of $(\|z_1\|,...,\|z_m\|)\in\mathbb{R}^{m}$.
Note that $T_K$ is a continuous function. 
With this function, our target optimization problem is formulated as 
\begin{align}\label{dcgl}
\underset{\bm{x}}{\mbox{minimize}} & \quad f(\bm{x})+\gamma T_K(D\bm{x}),
\end{align}
where $\gamma>0$, $f:\mathbb{R}^N\to\mathbb{R}$, and $D$ is a $pm\times N$ matrix. 
Note that if we set $f(\bm{x})=\sum_{i\in\mathcal{V}}f_i(x_i)$ and $D$ is a matrix such that $\bm{z}=D\bm{x}$ with $z_{\{i,j\}}=x_i-x_j$ for all $\{i,j\}\in\mathcal{E}$, then the problem \eqref{dcgl} is reduced to NTL~\eqref{dcnl}.

To apply ADMM, we rewrite the problem \eqref{dcgl} as the following equality-constrained formulation: 
\begin{align}
\underset{\bm{x},\bm{z}}{\mbox{minimize}} & \quad f(\bm{x})+\gamma T_K(\bm{z})
\label{obj:admmTK} \\
\text{subject to}                         & \quad \bm{z}=D\bm{x}.
\label{const:admmTK}
\end{align}

By introducing the dual variables $\bm{y}\in\mathbb{R}^{pm}$ for the equality constraints \eqref{const:admmTK}, the augmented Lagrangian function of \eqref{obj:admmTK}--\eqref{const:admmTK} is defined as
\begin{equation}
L_\rho(\bm{x},\bm{z},\bm{y})=f(\bm{x})+\gamma T_K(\bm{z})+\bm{y}^\top(\bm{z}-D\bm{x})+\frac{\rho}{2}{\|\bm{z}-D\bm{x}\|}_{2}^2,
\label{dcnal}
\end{equation}
with a positive constant $\rho$. 
ADMM is then described as Algorithm \ref{admm}.
\begin{algorithm}
\caption{ADMM for \eqref{dcgl}}
  \label{admm}
  \begin{algorithmic}
   \STATE {\bfseries Input:} $\bm{x}^0, \bm{y}^0, \rho>0$ and $t=0$.
   \REPEAT
   \STATE
\vspace{-7mm}
\begin{align}
   \bm{z}^{t+1}&\in\argmin_{\bm{z}}L_\rho(\bm{x}^t, \bm{z}, \bm{y}^t),\label{z-update}\\
   \bm{x}^{t+1}&\in\argmin_{\bm{x}}L_\rho(\bm{x}, \bm{z}^{t+1}, \bm{y}^t),\label{x-update}\\
   \bm{y}^{t+1}&=\bm{y}^t+\rho (\bm{z}^{t+1}-D\bm{x}^{t+1})\label{y-update}.
\end{align}
\vspace{-7mm}
   \STATE 
$t=t+1$
   \UNTIL Stopping criterion satisfied.
  \end{algorithmic}
\end{algorithm}
\subsection{Closed-form solution of Subproblem~\eqref{z-update}}
We can derive a closed-form solution of Subproblem~\eqref{z-update}. 
 First, it is easy to see that \eqref{z-update} is reduced to 
\begin{align}
\bm{z}^{t+1}\in
\mathrm{prox}_{\frac{\gamma}{\rho}T_K}(D\bm{x}^t-\frac{1}{\rho}\bm{y}^t)=
\argmin_{\bm{z}} &\Big\{ \frac{\gamma}{\rho}T_K(\bm{z})+\frac{1}{2}\|\bm{z}-(D\bm{x}^t-\frac{1}{\rho}\bm{y}^t)\|_2^2 \Big\},
\label{proximal}
\end{align}
where 
\[
\mathrm{prox}_f(\bm{x}):=\argmin_{\bm{z}}\Big\{f(\bm{z})+\frac{1}{2}\|\bm{z}-\bm{x}\|_2^2\Big\}
\]
is the proximal mapping of $\bm{x}$ with respect to $f$. 
Note that \eqref{proximal} may not be a singleton since $T_K$ is non-convex.

Though the minimization in \eqref{proximal} is a non-convex optimization, we can derive a closed-form solution, $\bm{z}^{t+1}$, in a similar manner to \citet{lu2018sparse} and \citet{bertsimas2017trimmed}.
For simplicity of notation, let $\bm{a}=D\bm{x}^t-\frac{1}{\rho}\bm{y}^t$. 
With this, the minimization in \eqref{proximal} can be equivalently rewritten as follows.
\begin{align}
\begin{split}
\underset{\bm{z}}{\min} \quad \gamma T_K(\bm{z})+\frac{\rho}{2}{\|\bm{z}-\bm{a}\|}_2^2 &= \quad \underset{\bm{z}}{\min} ~ \gamma\sum_{k=K+1}^{m}\|z_{(k)}\|_2+\frac{\rho}{2}\sum_{k=1}^m{\|z_k-a_k\|}_2^2 \\
  &= \quad \underset{\bm{z}}{\min} \quad \bigg\{\gamma\underset{\substack{I_k\in\{0, 1\}\\ \sum\limits_{k=1}^mI_k=m-K}}{\min}\Big\{\sum_{k=1}^m{\|z_k\|}_2I_k\Big\}+\frac{\rho}{2}\sum_{k=1}^m{\|z_k-a_k\|}_2^2\bigg\} \\
  &= \underset{\substack{I_k\in\{0, 1\}\\ \sum\limits_{k=1}^mI_k=m-K}}{\min} \bigg\{\underset{\bm{z}}{\min}\Big\{\gamma\sum_{k=1}^m{\|z_k\|}_2I_k+\frac{\rho}{2}\sum_{k=1}^m{\|z_k-a_k\|}_2^2\Big\}\bigg\} \\
  &= \underset{\substack{I_k\in\{0, 1\}\\ \sum\limits_{k=1}^mI_k=m-K}}{\min} \bigg\{\sum_{k=1}^m\underbrace{\underset{z_k}{\min}\left\{\gamma{\|z_k\|}_2I_k+\frac{\rho}{2}{\|z_k-a_k\|}_2^2\right\}}_{P_{(k)}}\bigg\},
\label{eq:prox_TK_end}
\end{split}
\end{align}
where the second equality is obtained by introducing integer variables $I_k$, which play a role as an indicator of the smallest $m-K$ components, and the third and fourth equalities are established by interchanging ``min" and ``min," or ``min" and ``summation," which is possible because of the separability with respect to $\bm{z}=(z_k)_{k\in[m]}$.
 For fixed $I_k$, we next evaluate the term 
\begin{align*}
P_{(k)}&:=\underset{z_k}{\min}\left\{\gamma{\|z_k\|}_2I_k+\frac{\rho}{2}{\|z_k-a_k\|}_2^2\right\}.
\end{align*}
To this end, let 
\[
P:=\underset{z}{\min}\Big\{\pi(z):=\gamma{\|z\|}_2\iota+\frac{\rho}{2}{\|z-a\|}_2^2\Big\}.
\]
for simplicity. 
Observe that when $\iota=0$, we have $\argmin_z \pi(z)=\{a\}$ and $P=0$; when $\iota=1$, we have
\begin{align*}
\argmin_z \pi(z)&=\mathrm{prox}_{\frac{\gamma}{\rho}{\|\cdot\|}_2}(a)=
\left\{
\begin{array}{cl}
 0,                                            & {\|a\|}_2\leq\frac{\gamma}{\rho}, \\
 \left(1-\frac{\gamma}{\rho{\|a\|}_2}\right)a, & {\|a\|}_2>\frac{\gamma}{\rho},
\end{array}
\right.
\end{align*}
and $P=\phi(\|a\|_2)$, 
where
\begin{align*}
\phi(t):=
\left\{
\begin{array}{cl}
 \frac{1}{2}t^2, & 0\leq t\leq\frac{\gamma}{\rho}, \\
 \frac{\gamma}{\rho}t-\frac{1}{2}{\left(\frac{\gamma}{\rho}\right)}^2, & t>\frac{\gamma}{\rho}.
\end{array}
\right.
\end{align*}
Accordingly, with $a=a_k$, the problem \eqref{eq:prox_TK_end} can be 
reduced to
\begin{align*}
\underset{\substack{I_k\in\{0, 1\}\\ \sum\limits_{k=1}^mI_k=m-K}}{\min} 
\sum_{k=1}^mP_{(k)}
&=
\underset{\substack{I_k\in\{0, 1\}\\ \sum\limits_{k=1}^mI_k=m-K}}{\min} 
\sum_{k=1}^mI_k\phi(\|a_k\|_2).
\end{align*}
Since $\phi(t)$ is increasing on $(0,\infty)$, an optimal solution of \eqref{proximal} is given by
\begin{align}
\begin{split}
z^{t+1}_k
&=\left\{
  \begin{array}{cl}
\displaystyle
    a_k, & \mbox{if }\|a_k\|_2\mbox{ is in the largest }K\mbox{ components of }(\|a_k\|_2)_{k\in[m]}, \\
\displaystyle
    \mathrm{prox}_{\frac{\gamma}{\rho}{\|\cdot\|}_2}(a_k), & \mbox{if }\|a_k\|_2\mbox{ is in the smallest }m-K\mbox{ components of }(\|a_k\|_2)_{k\in[m]}.
  \end{array}
  \right.
\end{split}
\label{eq:admm_z-formula}
\end{align}

\subsection{Proximal ADMM}
As for Subproblem~\eqref{x-update}, it is 
possible to derive a closed-form solution under restrictive assumptions (e.g., that of $f$ being a strictly convex quadratic function). 
However, it is often hard to obtain a closed-form solution.

To make the $\bm{x}$-update \eqref{x-update} at each iteration efficient, we consider \emph{Proximal ADMM} (\citet{li2015global}). 
Suppose that $f$ is $L$-smooth, so that the objective function of \eqref{x-update} is bounded above as
\begin{align*}
&L_\rho(\bm{x},\bm{z}^{t+1},\bm{y}^t)\\
&\leq f(\bm{x}^t)+\nabla f(\bm{x}^t)^\top(\bm{x}-\bm{x}^t)+\frac{L}{2}\|\bm{x}-\bm{x}^t\|_2^2
+(\bm{y}^t)^\top(\bm{z}^{t+1}-D\bm{x})+\frac{\rho}{2}\|\bm{z}^{t+1}-D\bm{x}\|_2^2
\end{align*}
by the inequality \eqref{inv_descent_lemma}.
The minimizer of the right-hand side is given by 
\begin{align}
\bm{x}^{t+1}=\left(I_N+\frac{\rho}{L}D^\top D\right)^{-1}\left(\bm{x}^{t}-\frac{1}{L}\nabla f(\bm{x}^{t})+\frac{1}{L}D^\top(\bm{y}^{t}+\rho\bm{z}^{t+1})\right),
\label{eq:closedform:x-step}
\end{align}
where $I_N$ is the $N$-dimensional identity matrix.
Note that the formula \eqref{eq:closedform:x-step} can be efficiently computed by a matrix-vector multiplication once the inverse on the right-hand side is fixed at the beginning of the algorithm. 

Proximal ADMM is equipped with a more general update rule that would include \eqref{eq:closedform:x-step} as a special case.
For a continuously differentiable function $\phi$ on $\mathbb{R}^N$, we define the Bregman distance 
of $\bm{x}$ and $\bm{x}'$ by
\begin{align}
B_\phi(\bm{x}, \bm{x}')=\phi(\bm{x})-\phi(\bm{x}')-\nabla\phi(\bm{x}')^\top(\bm{x}-\bm{x}').
\end{align}
In Proximal ADMM, $\bm{x}^{t+1}$ is updated by
\begin{align}\label{eq:prox:x-step}
\bm{x}^{t+1}&\in\argmin_{\bm{x}}\left\{L_\rho(\bm{x}, \bm{z}^{t+1}, \bm{y}^t)+B_\phi(\bm{x}, \bm{x}^t)\right\}
\end{align}
in place of \eqref{x-update}. 
If we employ $\phi(\bm{x})=\frac{L}{2}\|\bm{x}\|_2^2-f(\bm{x})$, \eqref{eq:closedform:x-step} and \eqref{eq:prox:x-step} are equivalent.
Algorithm \ref{prox-admm} is the description of Proximal ADMM, where the subroutine \eqref{eq:prox:x-step} is employed for $x$-update as well as the proximal mapping \eqref{eq:admm_z-formula} of $T_K$ for $z$-update. 
\begin{algorithm}[H]
  \caption{Proximal ADMM for \eqref{dcgl}}         
  \label{prox-admm}
  \begin{algorithmic}
   \STATE {\bfseries Input:} $\bm{x}^0, \bm{y}^0, \rho>0$, and $t=0$.
   \REPEAT
   \STATE Let $\bm{a}^t:=D\bm{x}^t-\frac{1}{\rho}\bm{y}^t$, then $\bm{z}^{t+1}$ is determined by \eqref{eq:admm_z-formula} (i.e., \eqref{z-update}).
   \STATE $\bm{x}^{t+1}$ is determined by \eqref{eq:prox:x-step}.
   \STATE $\bm{y}^{t+1}$ is determined by \eqref{y-update}.
   \UNTIL Stopping criterion satisfied.
  \end{algorithmic}
 \end{algorithm}
Note that when we set $\phi(\bm{x})=0$, Proximal ADMM is reduced to the ordinary ADMM (Algorithm \ref{admm}).

\subsection{Convergence of Proximal ADMM}
The main goal of this subsection is to show that under practical assumptions Proximal ADMM converges to a local minimum of \eqref{dcgl} with the general penalty term. 
To show the convergence, we first give a formula of the directional derivative of $T_K$, which is a generalization of the result for the case where $p=1$, given by \citet{amir2020trimmed}.

\begin{lemma}\label{lemma:d-diff}
Let $\Lambda_1=\{k\mid\|z_k\|_2<\|z_{(K)}\|_2\}$ and $\Lambda_2=\{k\mid\|z_k\|_2=\|z_{(K)}\|_2\}$. The directional derivative of $T_K$ at $\bm{z}\in\mathbb{R}^{pm}$ in the direction $\bm{v}\in\mathbb{R}^{pm}$ is given by
\begin{align}\label{eq:d-diff}
\mathrm{d}T_K(\bm{z};\bm{v})=\sum_{k\in\Lambda_1}\delta(z_k,v_k)^\top v_k+\min_{\substack{\Lambda\subset\Lambda_2 \\ |\Lambda|=m-K-|\Lambda_1|}}\sum_{k\in\Lambda}\delta(z_k,v_k)^\top v_k,
\end{align}
where
\begin{align}
\delta(z,v):=
\left\{
\begin{array}{cl}
 \frac{z}{\|z\|_2}, & z\neq0,\\
 \frac{v}{\|v\|_2}, & z=0, v\neq0,\\
 0, & z=0, v=0
\end{array}
\right.
\end{align}
and $\|z_{(0)}\|_2=\infty$.
\end{lemma}

\begin{proof}
First, note that the equation
\begin{align}\label{eq:anyrepTK}
T_K(\bm{z})=\sum_{k\in\Lambda_1}\|z_k\|_2+\sum_{k\in\Lambda}\|z_k\|_2
\end{align}
holds for any $\Lambda\subset\Lambda_2$ such that $|\Lambda|=m-K-|\Lambda_1|$. Let
\begin{align}
\Lambda_1^\eta:=\{k\mid\|z_k+\eta v_k\|_2<\|(z+\eta v)_{(K)}\|_2\},\\
\Lambda_2^\eta:=\{k\mid\|z_k+\eta v_k\|_2=\|(z+\eta v)_{(K)}\|_2\}.
\end{align}
Observe that there exists a positive number $\varepsilon$ such that $\|z_k+\eta v_k\|_2<\|(z+\eta v)_{(K)}\|_2$ for all $k\in\Lambda_1$ and $\|z_k+\eta v_k\|_2>\|(z+\eta v)_{(K)}\|_2$ for all $k\in(\Lambda_1\cup\Lambda_2)^c$ whenever $0<\eta<\varepsilon$ because of the continuity of $\ell_2$-norm. 
Hence $\Lambda_1\subset\Lambda_1^\eta$ and $\Lambda_1^\eta\cup\Lambda_2^\eta\subset\Lambda_1\cup\Lambda_2$ hold whenever $0<\eta<\varepsilon$. From this, we obtain
\begin{align}\label{eq:limrepTK}
T_K(\bm{z}+\eta\bm{v})=\sum_{k\in\Lambda_1}\|z_k+\eta v_k\|_2+\min_{\substack{\Lambda\subset\Lambda_2 \\ |\Lambda|=m-K-|\Lambda_1|}}\sum_{k\in\Lambda}\|z_k+\eta v_k\|_2,
\end{align}
for $\eta\in(0,\varepsilon)$. 
Combining \eqref{eq:anyrepTK} and \eqref{eq:limrepTK} yields
\begin{align}\label{eq:difrepTK}
T_K(\bm{z}+\eta\bm{v})=\sum_{k\in\Lambda_1}(\|z_k+\eta v_k\|_2-\|z_k\|_2)+\min_{\substack{\Lambda\subset\Lambda_2 \\ |\Lambda|=m-K-|\Lambda_1|}}\sum_{k\in\Lambda}(\|z_k+\eta v_k\|_2-\|z_k\|_2).
\end{align}
Furthermore, taking the limit $\eta\searrow0$, for any $k\in[m]$, we have
\begin{align}
\frac{\|z_k+\eta v_k\|_2-\|z_k\|_2}{\eta}\rightarrow
\left\{
\begin{array}{cl}
 \frac{z_k}{\|z_k\|_2}^\top v_k, & z_k\neq0,\\
 \|v_k\|_2, & z_k=0, v_k\neq0,\\
 0, & z_k=0, v_k=0,
\end{array}
\right.
\end{align}
that is, $\frac{\|z_k+\eta v_k\|_2-\|z_k\|_2}{\eta}\rightarrow\delta(z_k,v_k)^\top v_k$. Thus, we obtain
\begin{align}
\mathrm{d}T_K(\bm{z};\bm{v})
=&\lim_{\eta\searrow0}\frac{T_K(\bm{z}+\eta\bm{v})-T_K(\bm{z})}{\eta}\\
=&\lim_{\eta\searrow0}\frac{\sum\limits_{k\in\Lambda_1}(\|z_k+\eta v_k\|_2-\|z_k\|_2)}{\eta}+\lim_{\eta\searrow0}\frac{\min\limits_{\substack{\Lambda\subset\Lambda_2 \\ |\Lambda|=m-K-|\Lambda_1|}}\sum\limits_{k\in\Lambda}(\|z_k+\eta v_k\|_2-\|z_k\|_2)}{\eta}\\
=&\sum_{k\in\Lambda_1}\lim_{\eta\searrow0}\frac{(\|z_k+\eta v_k\|_2-\|z_k\|_2)}{\eta}+\min_{\substack{\Lambda\subset\Lambda_2 \\ |\Lambda|=m-K-|\Lambda_1|}}\sum_{k\in\Lambda}\lim_{\eta\searrow0}\frac{(\|z_k+\eta v_k\|_2-\|z_k\|_2)}{\eta}\\
=&\sum_{k\in\Lambda_1}\delta(z_k,v_k)^\top v_k+\min_{\substack{\Lambda\subset\Lambda_2 \\ |\Lambda|=m-K-|\Lambda_1|}}\sum_{k\in\Lambda}\delta(z_k,v_k)^\top v_k,
\end{align}
where the third equality is established by interchanging ``min" and ``limit," which is possible because $\{\Lambda\subset\Lambda_2\mid |\Lambda|=m-K-|\Lambda_1|\}$ is a finite set.
\end{proof}

The following result claims that stationary points and local minima of \eqref{dcgl} are equivalent in \eqref{dcgl} when $f$ is differentiable convex.

\begin{proposition}\label{prop:d-stational<->l-opt}
Suppose that $f$ is a differentiable convex function. If $\bm{x}^*$ is a directional-stationary point of \eqref{dcgl}, then 
it is locally optimal to \eqref{dcgl}.
\end{proposition}

\begin{proof}
To prove the proposition by contradiction, suppose that $\bm{x}^*$ is not a locally optimal solution of \eqref{dcgl}. 
Then there exists a sequence $\{\bm{x}^t\}$ such that $\bm{x}^t\rightarrow\bm{x}^*$ and $f(\bm{x}^*)+\gamma T_K(D\bm{x}^*)>f(\bm{x}^t)+\gamma T_K(D\bm{x}^t)$ for all $t$.
Setting
\begin{align}
\Lambda_1&:=\{k\mid\|(D\bm{x}^*)_k\|_2<\|(D\bm{x}^*)_{(K)}\|_2\},\\
\Lambda_2&:=\{k\mid\|(D\bm{x}^*)_k\|_2=\|(D\bm{x}^*)_{(K)}\|_2\},\\
\Lambda_1^t&:=\{k\mid\|(D\bm{x}^t)_k\|_2<\|(D\bm{x}^t)_{(K)}\|_2\},\\
\Lambda_2^t&:=\{k\mid\|(D\bm{x}^t)_k\|_2=\|(D\bm{x}^t)_{(K)}\|_2\},
\end{align}
we have
\begin{align}
\lefteqn{T_K(D\bm{x}^t)-T_K(D\bm{x}^*)}\\
&=\sum_{k\in\Lambda_1}(\|(D\bm{x}^t)_k\|_2-\|(D\bm{x}^*)_k\|_2)+\min_{\substack{\Lambda\subset\Lambda_2 \\ |\Lambda|=m-K-|\Lambda_1|}}\sum_{k\in\Lambda}(\|(D\bm{x}^t)_k\|_2-\|(D\bm{x}^*)_k\|_2),
\end{align}
since $\Lambda_1\subset\Lambda_1^t$ and $\Lambda_1^t\cup\Lambda_2^t\subset\Lambda_1\cup\Lambda_2$ hold for sufficiently large $t$ as in the proof of Lemma \ref{lemma:d-diff}.
Noting that for any $z, z'\in\mathbb{R}^p$,
\begin{align}
\|z\|_2-\|z'\|_2\geq\delta(z',z-z')^\top(z-z'),
\end{align}
we have
\begin{align}
&T_K(D\bm{x}^t)-T_K(D\bm{x}^*)\\
&\geq\sum_{k\in\Lambda_1}\delta((D\bm{x}^*)_k,(D\bm{v})_k)^\top(D\bm{v})_k+\min_{\substack{\Lambda\subset\Lambda_2 \\ |\Lambda|=m-K-|\Lambda_1|}}\sum_{k\in\Lambda}\delta((D\bm{x}^*)_k,(D\bm{v})_k)^\top(D\bm{v})_k,
\end{align}
where $\bm{v}=\bm{x}^t-\bm{x}^*$.
This as well as the convexity of $f$ and Lemma \ref{lemma:d-diff} yield
\begin{align}
0 &>f(\bm{x}^t)+\gamma T_K(D\bm{x}^t)-(f(\bm{x}^*)+\gamma T_K(D\bm{x}^*))\\
&\geq\nabla f(\bm{x}^*)^\top\bm{v}+\gamma\Big[\sum_{k\in\Lambda_1}\delta((D\bm{x}^*)_k,(D\bm{v})_k)^\top(D\bm{v})_k+
\min_{\substack{\Lambda\subset\Lambda_2 \\ |\Lambda|=m-K-|\Lambda_1|}}
\sum_{k\in\Lambda}\delta((D\bm{x}^*)_k,(D\bm{v})_k)^\top(D\bm{v})_k\Big]\\
&=\nabla f(\bm{x}^*)^\top\bm{v}+\gamma \mathrm{d}T_K(D\bm{x}^*;D\bm{v})\\
&=\mathrm{d}(f+\gamma T_K\circ D)(\bm{x}^*;\bm{v}),
\end{align}
which contradicts the fact that $\bm{x}^*$ is a stationary point of \eqref{dcgl}.
\end{proof}

The rest of this subsection is devoted to convergence results of Proximal ADMM, for which proofs are based on ideas of \citet{li2015global}.
The differences between their results and ours are summarized as follow:
\begin{itemize}
\item To apply Proposition \ref{prop:d-stational<->l-opt}, we will prove the convergence to a directional-stationary point. On the other hand, they prove convergence to a \emph{limiting-stationary point}, which is a weaker stationary point than a directional-stationary point (see e.g., \citet[pp.3350--3351]{cui2018composite}).
\item They assume the second-order differentiability of $f$, while we only assume the first-order differentiability of $f$.
\end{itemize}

Let us start with a result under a bit stronger assumption. 
\begin{proposition}\label{prop:weak-convergence}
Suppose that $f$ is convex, and $f$ and $\phi$ are continuously differentiable. If the sequence $\{(\bm{x}^t,\bm{z}^t,\bm{y}^t)\}$ generated from Proximal ADMM has a partial limit $(\bm{x}^*,\bm{z}^*,\bm{y}^*)$ and $(\|\bm{x}^{t+1}-\bm{x}^t\|_2,\|\bm{z}^{t+1}-\bm{z}^t\|_2,\|\bm{y}^{t+1}-\bm{y}^t\|_2)$ converges to $(0,0,0)$, then $\bm{x}^*$ is a local minimum of \eqref{dcgl}.
\end{proposition}

Note that from Proposition \ref{prop:weak-convergence}, if the whole sequence $\{(\bm{x}^t,\bm{z}^t,\bm{y}^t)\}$ converges to $(\bm{x}^*,\bm{z}^*,\bm{y}^*)$, then $\bm{x}^*$ is a local minimum of \eqref{dcgl}.

\begin{proof}
Let $\{(\bm{x}^{t_i},\bm{z}^{t_i},\bm{y}^{t_i})\}$ be a subsequence of $\{(\bm{x}^t,\bm{z}^t,\bm{y}^t)\}$ that converges to $(\bm{x}^*,\bm{z}^*,\bm{y}^*)$.
From the fact that $(\|\bm{x}^{t+1}-\bm{x}^t\|_2,\|\bm{z}^{t+1}-\bm{z}^t\|_2,\|\bm{y}^{t+1}-\bm{y}^t\|_2)$ converges to $(0,0,0)$, the subsequence $\{(\bm{x}^{t_i+1},\bm{z}^{t_i+1},\bm{y}^{t_i+1})\}$ also converges to $(\bm{x}^*,\bm{z}^*,\bm{y}^*)$.
By the relation \eqref{y-update}, the equation
\begin{align}
\bm{y}^{t_i+1}&=\bm{y}^{t_i}+\rho (\bm{z}^{t_i+1}-D\bm{x}^{t_i+1})
\end{align}
holds. Letting $i\rightarrow\infty$ yields
\begin{align}\label{eq:lim-eq-3}
\bm{z}^*=D\bm{x}^*.
\end{align}
Taking the limit of the optimality condition of \eqref{eq:prox:x-step}, we have
\begin{align}
\nabla f(\bm{x}^{t_i+1})+\rho D^\top\left(D\bm{x}^{t_i+1}-\bm{z}^{t_i+1}-\frac{1}{\rho}\bm{y}^{t_i}\right)+\nabla\phi(\bm{x}^{t_i+1})-\nabla\phi(\bm{x}^{t_i})=0,
\end{align}
and combining it with \eqref{eq:lim-eq-3} and continuity of $\nabla f$ and $\nabla\phi$, we obtain
\begin{align}\label{eq:lim-eq-2}
\nabla f(\bm{x}^*)=D^\top\bm{y}^*.
\end{align}
Since $\bm{z}^{t_i+1}$ is optimal to \eqref{z-update}, the inequality
\begin{align}\label{eq:eq-1}
\begin{split}
&\gamma T_K(\bm{z}^{t_i+1})+(\bm{y}^{t_i})^\top\bm{z}^{t_i+1}+\frac{\rho}{2}\|\bm{z}^{t_i+1}-D\bm{x}^{t_i}\|_2^2\\
&\leq \gamma T_K(\bm{z}^*+\eta D\bm{v})+(\bm{y}^{t_i})^\top(\bm{z}^*+\eta D\bm{v})+\frac{\rho}{2}\|\bm{z}^*+\eta D\bm{v}-D\bm{x}^{t_i}\|_2^2
\end{split}
\end{align}
holds for any $\eta>0$ and $\bm{v}\in\mathbb{R}^N$.
By the continuity of $T_K$ and \eqref{eq:lim-eq-3}, letting $i\rightarrow\infty$ yields
\begin{align}\label{eq:lim-eq-1}
\begin{split}
\gamma T_K(D\bm{x}^*)+(\bm{y}^*)^\top D\bm{x}^* \leq \gamma T_K(D\bm{x}^*+\eta D\bm{v})+(\bm{y}^*)^\top(D\bm{x}^*+\eta D\bm{v})+\frac{\rho}{2}\|\eta D\bm{v}\|_2^2.
\end{split}
\end{align}
Combining this with \eqref{eq:lim-eq-2}, we see that
\begin{align}\label{eq:pre-d-stational}
\begin{split}
&\eta\nabla f(\bm{x}^*)^\top\bm{v}+\gamma T_K(D(\bm{x}^*+\eta\bm{v}))-\gamma T_K(D\bm{x}^*)+\eta^2\frac{\rho}{2}\|D\bm{v}\|_2^2\\
&= \eta(D^\top\bm{y}^*)^\top\bm{v}+\gamma T_K(D\bm{x}^*+\eta D\bm{v})-\gamma T_K(D\bm{x}^*)+\eta^2\frac{\rho}{2}\|D\bm{v}\|_2^2\\
&= (\bm{y}^*)^\top(\eta D\bm{v})+\gamma T_K(D\bm{x}^*+\eta D\bm{v})-\gamma T_K(D\bm{x}^*)+\frac{\rho}{2}\|\eta D\bm{v}\|_2^2\\
&\geq 0.
\end{split}
\end{align}
By dividing both sides of this inequality by $\eta$ and taking the limit with $\eta\searrow0$, we obtain
\begin{align}\label{eq:d-stational}
\begin{split}
\mathrm{d}(f+\gamma T_K\circ D)(\bm{x}^*;\bm{v})= \nabla f(\bm{x}^*)^\top\bm{v}+\gamma \mathrm{d}(T_K\circ D)(\bm{x}^*;\bm{v})\geq 0,
\end{split}
\end{align}
which implies that $\bm{x}^*$ is a stationary point of \eqref{dcgl}. Since $f$ is a differentiable convex function, $\bm{x}^*$ is shown to be locally optimal to \eqref{dcgl} by Proposition \ref{prop:d-stational<->l-opt}.
\end{proof}

By adding assumptions, we have a stronger convergence result than Proposition \ref{prop:weak-convergence}.

\begin{theorem}\label{thm:strong-convergence}
Suppose that the following assumptions hold:
\begin{enumerate}[({A}1)]
\item $D$ is surjective, that is, $\sigma:=\lambda_{\min}(DD^\top)>0$;
\item $f$ is convex;
\item $f+\phi$ is $L_1$-smooth;
\item $f+\phi+\frac{\rho}{2}\|D\cdot\|_2^2$ is $\alpha_1$-strongly convex;
\item $\phi$ is $L_2$-smooth and $\alpha_2$-strongly convex;
\item There exists $0<r<1$ such that $\rho>\frac{2}{\sigma(\alpha_1+\alpha_2)}\left(\frac{L_1^2}{r}+\frac{L_2^2}{1-r}\right)$,
\end{enumerate}
where we allow $L_1, L_2, \alpha_1, \alpha_2$ to be $0$, but we must have $\alpha_1+\alpha_2>0$.
If the sequence $\{(\bm{x}^t,\bm{z}^t,\bm{y}^t)\}$ generated from Proximal ADMM has a partial limit $(\bm{x}^*,\bm{z}^*,\bm{y}^*)$, then $\bm{x}^*$ is a local minimum of \eqref{dcgl}.
\end{theorem}

\begin{proof}
From the optimality condition of \eqref{eq:prox:x-step} and the equation \eqref{y-update}, we obtain
\begin{align}\label{eq:y-b-x}
\begin{split}
\sigma\|\bm{y}^{t+1}-\bm{y}^{t}\|_2^2 &\leq\|D^\top(\bm{y}^{t+1}-\bm{y}^{t})\|_2^2\\
&=\|\nabla f(\bm{x}^{t+1})+\nabla\phi(\bm{x}^{t+1})-\nabla\phi(\bm{x}^t)-\nabla f(\bm{x}^t)-\nabla\phi(\bm{x}^t)+\nabla\phi(\bm{x}^{t-1})\|_2^2\\
&\leq\frac{1}{r}\|\nabla f(\bm{x}^{t+1})+\nabla\phi(\bm{x}^{t+1})-\nabla f(\bm{x}^t)-\nabla\phi(\bm{x}^t)\|_2^2+\frac{1}{1-r}\|\nabla\phi(\bm{x}^t)-\nabla\phi(\bm{x}^{t-1})\|_2^2\\
&\leq\frac{L_1^2}{r}\|\bm{x}^{t+1}-\bm{x}^t\|_2^2+\frac{L_2^2}{1-r}\|\bm{x}^t-\bm{x}^{t-1}\|_2^2,
\end{split}
\end{align}
where the number $r$ satisfies the assumption \emph{(A6)} and the first inequality follows from the assumption \emph{(A1)}, the second one from the inequality $\|a+b\|_2^2\le\frac{\|a\|_2^2}{r}+\frac{\|b\|_2^2}{1-r}$, the third one from the assumptions \emph{(A3)} and \emph{(A5)}.
On the other hand, combining the equation \eqref{y-update} with the triangle inequality yields
\begin{align}\label{eq:z-b-x}
\|\bm{z}^{t+1}-\bm{z}^t\|_2\leq\|D(\bm{x}^{t+1}-\bm{x}^t)\|_2+\frac{1}{\rho}\|\bm{y}^{t+1}-\bm{y}^t\|_2+\frac{1}{\rho}\|\bm{y}^t-\bm{y}^{t-1}\|_2.
\end{align}
The above two inequalities imply that if the sequence $\|\bm{x}^{t+1}-\bm{x}^t\|_2$ converges to $0$, then both $\|\bm{z}^{t+1}-\bm{z}^t\|_2$ and $\|\bm{y}^{t+1}-\bm{y}^t\|_2$ also converge to $0$.
Thus we next show that $\|\bm{x}^{t+1}-\bm{x}^t\|_2$ converges to $0$.

From \eqref{y-update} and \eqref{eq:y-b-x}, we obtain
\begin{align}
\begin{split}
L_\rho(\bm{x}^{t+1},\bm{z}^{t+1},\bm{y}^{t+1})-L_\rho(\bm{x}^{t+1},\bm{z}^{t+1},\bm{y}^t) &=(\bm{y}^{t+1}-\bm{y}^t)^\top(\bm{z}^{t+1}-D\bm{x}^{t+1})\\
&=\frac{1}{\rho}\|\bm{y}^{t+1}-\bm{y}^{t}\|_2^2\\
&\leq\frac{L_1^2}{\sigma\rho r}\|\bm{x}^{t+1}-\bm{x}^t\|_2^2+\frac{L_2^2}{\sigma\rho(1-r)}\|\bm{x}^t-\bm{x}^{t-1}\|_2^2.
\end{split}
\end{align}
Since $L_\rho(\bm{x}, \bm{z}^{t+1}, \bm{y}^t)+B_\phi(\bm{x}, \bm{x}^t)$ is $\alpha_1$-strongly convex by the assumption \emph{(A4)}, using the inequality \eqref{plain_bound}, we have
\begin{align}
\begin{split}
L_\rho(\bm{x}^{t+1},\bm{z}^{t+1},\bm{y}^t)-L_\rho(\bm{x}^t,\bm{z}^{t+1},\bm{y}^t) &\leq-\frac{\alpha_1}{2}\|\bm{x}^{t+1}-\bm{x}^t\|_2^2-D_\phi(\bm{x}^{t+1},\bm{x}^t)\\
&\leq-\frac{\alpha_1}{2}\|\bm{x}^{t+1}-\bm{x}^t\|_2^2-\frac{\alpha_2}{2}\|\bm{x}^{t+1}-\bm{x}^t\|_2^2\\
&=-\frac{\alpha_1+\alpha_2}{2}\|\bm{x}^{t+1}-\bm{x}^t\|_2^2,
\end{split}
\end{align}
where we use the $\alpha_2$-strong convexity of $\phi$ (the assumption \emph{(A5)}) in the second inequality.
Furthermore, because $\bm{z}^{t+1}$ is a minimizer of \eqref{z-update}, the inequality
\begin{align}
L_\rho(\bm{x}^t,\bm{z}^{t+1},\bm{y}^t)-L_\rho(\bm{x}^t,\bm{z}^t,\bm{y}^t) &\leq0
\end{align}
holds.
By adding the above three inequalities together, we have
\begin{align}\label{eq:once-lyapunov}
L_\rho(\bm{x}^{t+1},\bm{z}^{t+1},\bm{y}^{t+1})-L_\rho(\bm{x}^t,\bm{z}^t,\bm{y}^t) &\leq\left(\frac{L_1^2}{\sigma\rho r}-\frac{\alpha_1+\alpha_2}{2}\right)\|\bm{x}^{t+1}-\bm{x}^t\|_2^2+\frac{L_2^2}{\sigma\rho(1-r)}\|\bm{x}^t-\bm{x}^{t-1}\|_2^2.
\end{align}
Let $\{(\bm{x}^{t_i},\bm{z}^{t_i},\bm{y}^{t_i})\}$ be a subsequence of $\{(\bm{x}^t,\bm{z}^t,\bm{y}^t)\}$ that converges to a partial limit $(\bm{x}^*,\bm{z}^*,\bm{y}^*)$.
Noting that $C:=\frac{\alpha_1+\alpha_2}{2}-\frac{1}{\sigma\rho}\left(\frac{L_1^2}{r}+\frac{L_2^2}{1-r}\right)>0$ from the assumption \emph{(A6)}, we have
\begin{align}\label{eq:lyapunov}
\begin{split}
&L_\rho(\bm{x}^{t_i},\bm{z}^{t_i},\bm{y}^{t_i})-L_\rho(\bm{x}^1,\bm{z}^1,\bm{y}^1)\\
&\leq \left\{\sum_{t=1}^{t_i-1}\left(\frac{L_1^2}{\sigma\rho r}-\frac{\alpha_1+\alpha_2}{2}\right)\|\bm{x}^{t+1}-\bm{x}^t\|_2^2+\frac{L_2^2}{\sigma\rho(1-r)}\|\bm{x}^t-\bm{x}^{t-1}\|_2^2\right\}\\
&= \sum_{t=1}^{t_i-1}\left(\frac{L_1^2}{\sigma\rho r}-\frac{\alpha_1+\alpha_2}{2}\right)\|\bm{x}^{t+1}-\bm{x}^t\|_2^2+\sum_{t=0}^{t_i-2}\frac{L_2^2}{\sigma\rho(1-r)}\|\bm{x}^{t+1}-\bm{x}^t\|_2^2\\
&= \frac{L_2^2}{\sigma\rho(1-r)}\|\bm{x}^1-\bm{x}^0\|_2^2-\left(\frac{\alpha_1+\alpha_2}{2}-\frac{L_1^2}{\sigma\rho r}\right)\|\bm{x}^{t_i}-\bm{x}^{t_i-1}\|_2^2-\sum_{t=1}^{t_i-1}C\|\bm{x}^{t+1}-\bm{x}^t\|_2^2\\
&\leq \frac{L_2^2}{\sigma\rho(1-r)}\|\bm{x}^1-\bm{x}^0\|_2^2-C\sum_{t=1}^{t_i-1}\|\bm{x}^{t+1}-\bm{x}^t\|_2^2.
\end{split}
\end{align}
By the continuity of $L_\rho$, we have 
\begin{align}\label{eq:lyapunov-val}
\lim_{i\rightarrow\infty}L_\rho(\bm{x}^{t_i},\bm{z}^{t_i},\bm{y}^{t_i})= L_\rho(\bm{x}^*,\bm{z}^*,\bm{y}^*)>-\infty.
\end{align}
Taking the limit $i\rightarrow\infty$ in \eqref{eq:lyapunov} with \eqref{eq:lyapunov-val} leads to
\begin{align}
\sum_{t=1}^{\infty}\|\bm{x}^{t+1}-\bm{x}^t\|_2^2<\infty,
\end{align}
which implies that $\|\bm{x}^{t+1}-\bm{x}^t\|_2\rightarrow0$. Thus, $(\|\bm{x}^{t+1}-\bm{x}^t\|_2,\|\bm{z}^{t+1}-\bm{z}^t\|_2,\|\bm{y}^{t+1}-\bm{y}^t\|_2)\rightarrow(0,0,0)$ holds. Since $f$ and $\phi$ are continuously differentiable and $f$ is a convex function by the assumptions \emph{(A2)} and \emph{(A3)}, 
Proposition \ref{prop:weak-convergence} yields the desired result.
\end{proof}

We will note below how to choose $\phi$ and $\rho$ based on Theorem~\ref{thm:strong-convergence}.

\begin{example}\label{ex:ordinary}
Consider the case where $D$ is surjective, $f$ is $L$-smooth and $\alpha$-strongly convex, and $\phi=0$. 
The assumptions (A3)--(A5) are then fulfilled with $L_1=L$, $L_2=0$, $\alpha_1=\alpha$, and $\alpha_2=0$. 
By choosing $\rho>\frac{2L^2}{\sigma\alpha r}$ for some $0<r<1$, the assumption (A6) is also fulfilled. 
\end{example}

\begin{example}\label{ex:smooth}
Consider the case where $D$ is surjective and $f$ is $L$-smooth and convex, and $\phi=\frac{L}{2}\|\cdot\|_2^2-f$.
In this case, $\phi$ is L-smooth since $\phi$ is differentiable convex and $\frac{L}{2}\|\cdot\|_2^2-\phi=f$ is convex.
Accordingly, the assumptions (A3)--(A5) are then fulfilled with $L_1=L_2=\alpha_1=L$ and $\alpha_2=0$.
By choosing $\rho>\frac{2L}{\sigma}\left(\frac{1}{r}+\frac{1}{1-r}\right)$ for some $0<r<1$, the assumption (A6) also holds.
\end{example}

\begin{example}\label{ex:non-surjective}
When $D$ is not surjective, we cannot apply Theorem \ref{thm:strong-convergence} because $\sigma:=\lambda_{\min}(DD^\top)=0$. It implies that $\rho>\frac{2}{\sigma(\alpha_1+\alpha_2)}\left(\frac{L_1^2}{r}+\frac{L_2^2}{1-r}\right)=\infty$ as a formality. In the computational examples of Section \ref{sec:numerical}, we choose $\phi$ such that the assumptions (A3)--(A5) hold and take a large $\rho$ to mitigate the inconsistency.
\end{example}

While Theorem \ref{thm:strong-convergence} assumes that Proximal ADMM has a partial limit, 
the existence of a partial limit is guaranteed by the following theorem.

\begin{theorem}\label{thm:boundness}
In addition to the assumptions (A1), (A3)--(A6), suppose that $f$ is coercive, i.e., $\lim_{\|\bm{x}\|_2\rightarrow\infty}f(\bm{x})=\infty$, and that there exists $0<\zeta<\sigma\rho r$ such that
\begin{align}
f_{\inf}:=\inf_{\bm{x}}\left\{f(\bm{x})-\frac{1}{2\zeta}\|\nabla f(\bm{x})\|_2^2\right\}>-\infty.
\end{align}
Then the sequence $\{(\bm{x}^t,\bm{z}^t,\bm{y}^t)\}$ generated from Proximal ADMM is bounded.
\end{theorem}

\begin{proof}
Since the assumptions \emph{(A1)}, \emph{(A3)}--\emph{(A6)} are fulfilled, the inequality \eqref{eq:once-lyapunov} holds. 
By slightly transforming it, we obtain
\begin{align}
L_\rho(\bm{x}^{t+1},\bm{z}^{t+1},\bm{y}^{t+1})&+\frac{L_2^2}{\sigma\rho(1-r)}\|\bm{x}^{t+1}-\bm{x}^t\|_2^2-L_\rho(\bm{x}^t,\bm{z}^t,\bm{y}^t)-\frac{L_2^2}{\sigma\rho(1-r)}\|\bm{x}^t-\bm{x}^{t-1}\|_2^2\\
&\leq\left(\frac{L_1^2}{\sigma\rho r}+\frac{L_2^2}{\sigma\rho(1-r)}-\frac{\alpha_1+\alpha_2}{2}\right)\|\bm{x}^{t+1}-\bm{x}^t\|_2^2\\
&\leq0,
\end{align}
which implies that the sequence $L_\rho(\bm{x}^t,\bm{z}^t,\bm{y}^t)+\frac{L_2^2}{\sigma\rho(1-r)}\|\bm{x}^t-\bm{x}^{t-1}\|_2^2$ is monotonically decreasing. Hence, we see that
\begin{align}\label{eq:monotone}
&L_\rho(\bm{x}^t,\bm{z}^t,\bm{y}^t)+\frac{L_2^2}{\sigma\rho(1-r)}\|\bm{x}^t-\bm{x}^{t-1}\|_2^2\le L_\rho(\bm{x}^1,\bm{z}^1,\bm{y}^1)+\frac{L_2^2}{\sigma\rho(1-r)}\|\bm{x}^1-\bm{x}^0\|_2^2.
\end{align}
On the other hand, combining \eqref{y-update} and the optimality condition of \eqref{eq:prox:x-step} yields
\begin{align}
\nabla f(\bm{x}^t)-D^\top\bm{y}^t+\nabla\phi(\bm{x}^t)-\nabla\phi(\bm{x}^{t-1})=0.
\end{align}
Then, from the assumptions \emph{(A1)} and \emph{(A5)}, we have
\begin{align}\label{eq:y-bound-nablaf}
\begin{split}
\sigma\|\bm{y}^t\|_2^2 &\le\|D^\top\bm{y}^t\|_2^2\\
&\le\|\nabla f(\bm{x}^t)+\nabla\phi(\bm{x}^t)-\nabla\phi(\bm{x}^{t-1})\|_2^2\\
&\le\frac{1}{r}\|\nabla f(\bm{x}^t)\|_2^2+\frac{1}{1-r}\|\nabla\phi(\bm{x}^t)-\nabla\phi(\bm{x}^{t-1})\|_2^2\\
&\le\frac{1}{r}\|\nabla f(\bm{x}^t)\|_2^2+\frac{L_2^2}{1-r}\|\bm{x}^t-\bm{x}^{t-1}\|_2^2.
\end{split}
\end{align}
Combining \eqref{eq:monotone} with \eqref{eq:y-bound-nablaf} shows that
\begin{align}
&L_\rho(\bm{x}^1,\bm{z}^1,\bm{y}^1)+\frac{L_2^2}{\sigma\rho(1-r)}\|\bm{x}^1-\bm{x}^0\|_2^2\\
&\ge f(\bm{x}^t)+\gamma T_K(\bm{z}^t)+{\bm{y}^t}^\top(\bm{z}^t-D\bm{x}^t)+\frac{\rho}{2}{\|\bm{z}^t-D\bm{x}^t\|}_{2}^2+\frac{L_2^2}{\sigma\rho(1-r)}\|\bm{x}^t-\bm{x}^{t-1}\|_2^2\\
&=f(\bm{x}^t)+\gamma T_K(\bm{z}^t)+\frac{\rho}{2}{\left\|\bm{z}^t-D\bm{x}^t+\frac{1}{\rho}\bm{y}^t\right\|}_{2}^2-\frac{1}{2\rho}\|\bm{y}^t\|_2^2+\frac{L_2^2}{\sigma\rho(1-r)}\|\bm{x}^t-\bm{x}^{t-1}\|_2^2\\
&\ge f(\bm{x}^t)-\frac{1}{2\sigma\rho r}\|\nabla f(\bm{x}^t)\|_2^2-\frac{L_2^2}{2\sigma\rho(1-r)}\|\bm{x}^t-\bm{x}^{t-1}\|_2^2+\frac{L_2^2}{\sigma\rho(1-r)}\|\bm{x}^t-\bm{x}^{t-1}\|_2^2\\
&\ge\left(1-\frac{\zeta}{\sigma\rho r}\right)f(\bm{x}^t)+\frac{\zeta}{\sigma\rho r}\left\{f(\bm{x}^t)-\frac{1}{2\zeta}\|\nabla f(\bm{x}^t)\|_2^2\right\}\\
&\ge\left(1-\frac{\zeta}{\sigma\rho r}\right)f(\bm{x}^t)+\frac{\zeta}{\sigma\rho r}f_{\inf}.
\end{align}
Since $f$ is coercive, the above inequality implies that $\{\bm{x}^t\}$ is bounded.
The boundedness of $\{\bm{y}^t\}$ and $\{\bm{z}^t\}$ follows from \eqref{eq:y-bound-nablaf} and \eqref{y-update}, respectively.
\end{proof}

\begin{example}\label{ex:boundness}
If $f$ is $L$-smooth and bounded below, then the inequality
\begin{align}
\inf_{\bm{x}}\Big\{f(\bm{x})-\frac{1}{2L}\|\nabla f(\bm{x})\|_2^2\Big\}>-\infty.
\end{align}
holds (see \citet[Remark 3]{li2015global}).
Note that a continuous and coercive function is bounded below. If $f$ is $L$-smooth and coercive, we choose $\rho$ so that it satisfies not only the inequality in Example \ref{ex:ordinary} or \ref{ex:smooth}, but also the condition $\rho>\frac{L}{\sigma r}$.
\end{example}

\subsection{Computation of cluster path}
To get a cluster path of NTL~\eqref{dcnl} on the basis of Proximal ADMM, we use a warm start. Let $\{K_t\}_{t=1}^T\subset\{0, 1,\ldots,\left|\mathcal{E}\right|\}$ be a decreasing sequence of the cardinality parameter $K$.
 \begin{algorithm}[H]
  \caption{Cluster Path}         
  \label{clusterpath}
  \begin{algorithmic}
   \STATE {\bfseries Input:} $\bm{x}^0, \bm{y}^0=0, \rho>0, \{K_t\}_{t=1}^T$.
   \FOR{$t = 1$ to $T$}
   \STATE Get $\bm{x}^t$ by using Proximal ADMM to solve (\ref{dcnl}) with $K_t$, $\bm{x}^{t-1}$, and $\bm{y}^{0}$.
   \ENDFOR
  \end{algorithmic}
 \end{algorithm}
Note that when $f_i$ is convex for all $i\in\mathcal{V}$, NL is a convex optimization problem and a global optimum is attained by any local search method, but NTL is a non-convex optimization, and the output of Proximal ADMM is expected to be very sensitive to the initial point $(\bm{x}^0,\bm{y}^0)$ (and $K$).
The choice of the initial point of cluster path is discussed through numerical experiments in Section \ref{sec:numerical}.

\section{Numerical examples}
\label{sec:numerical}
This section presents several numerical examples to demonstrate how NTL behaves in comparison with NL.
We used ADMM (Algorithm \ref{admm}) to solve NTL and Algorithm \ref{clusterpath} to generate a cluster path.
For ADMM to solve NTL, we used the following termination condition: $\|\bm{z}^{t+1}-D\bm{x}^{t+1}\|_2\leq\sqrt{p|\mathcal{E}|}\varepsilon^{\rm abs}+\varepsilon^{\rm rel}\max\{\|\bm{z}^{t+1}\|_2, \|D\bm{x}^{t+1}\|_2\}$ and $\|\bm{x}^{t+1}-\bm{x}^{t}\|_2\leq\sqrt{pn}\varepsilon^{\rm abs}+\varepsilon^{\rm rel}\|\bm{x}^{t+1}\|_2$ is satisfied with $\varepsilon^{\rm abs}=\varepsilon^{\rm rel}=10^{-5}$, or the number of iterations reaches 1000.
For NL, we also used ADMM and increased $\gamma$ as described in Section 1 when generating a cluster path.
ADMM for NL was terminated either when $\|\bm{z}^{t+1}-D\bm{x}^{t+1}\|_2\leq\sqrt{p|\mathcal{E}|}\varepsilon^{\rm abs}+\varepsilon^{\rm rel}\max\{\|\bm{z}^{t+1}\|_2, \|D\bm{x}^{t+1}\|_2\}$ and $\rho\|D(\bm{x}^{t+1}-\bm{x}^{t})\|_2\leq\sqrt{p|\mathcal{E}|}\varepsilon^{\rm abs}+\varepsilon^{\rm rel}\|\bm{y}^{t+1}\|_2$ were satisfied (as appeared in \citet{boyd2011distributed}) for $\varepsilon^{\rm abs}=\varepsilon^{\rm rel}=10^{-5}$, or when the number of iterations reached 1000.

\subsection{Ridge regression under two latent clusters}
We first consider a case where there is no prior information, that is, $W\equiv1$.
We solved \eqref{nlasso} and \eqref{dcnl}, respectively, for simple regression models using two datasets presented in the top row of Figure \ref{visual2}, where the number of data points is $n=100$ and the two latent clusters correspond to red and blue points.
Obviously, the left-hand side panel is the case where regression lines have different slopes and the dataset has a clear cluster structure, while in the right-hand side panel the two regression lines have similar slopes while keeping the linear separability of the two clouds. 
For each data point $(a_i,b_i)\in\mathbb{R}^2$, we consider the loss function of the form:
\begin{align}
f_i(x_{i,1},x_{i,2})&=\frac{1}{2}{\|b_i-x_{i,1}-a_ix_{i,2}\|}_2^2+\frac{\varepsilon}{2}x_{i,2}^2,
\quad i=1,...,100,
\label{ridge}
\end{align}
where $x_{i,1}$ and $x_{i,2}$ are the intercept and the slope, respectively, of the model corresponding to data point $i$, 
and 
$\varepsilon>0$ is a parameter to trade-off between the squared residual and the $\ell_2$-regularizer.
In this experiment we set $\varepsilon={10}^{-2}$ and consider a complete graph, i.e., $\mathcal{E}=\left\{\{i,j\}\mid i\neq j, i,j\in\mathcal{V}\right\}$,
and uniform weights $w_{\{i,j\}}=1$ for all $i,j\in\mathcal{V}$. Initial points of cluster path are defined by $\bm{y}^0=0, x_i^0=\argmin_{x\in\mathbb{R}^{2}}f_i(x)$.

For NL, let $(x_{i,1}(\gamma),x_{i,2}(\gamma))$ denote the centroid of data point $i$, obtained by ADMM under parameter $\gamma$.
The second row of Figure \ref{visual2} shows the cluster paths of centroids, $\{(x_{i,1}(\gamma),x_{i,2}(\gamma)):\gamma=10^{-3}\times(1.2)^{t-1},t=1,...,50\}$, generated by NL (via ADMM) with $\gamma$ increasing.
Starting with the initial points, $(x_{i,1}(0),x_{i,2}(0))=(b_i,0)=x_i^0$, which are highlighted in red or blue, they converge to a single black point in the middle as $\gamma$ grows. 
We can see, however, from these two panels that NL failed to capture the cluster structure well for either data sets in that the loci of centroids kept separated until when only one cluster was formed at the center point with a sufficiently large $\gamma$.

The third row of Figure \ref{visual2} shows cluster paths generated by Algorithm \ref{clusterpath}. We employed the same initial points as in NL.
From Example \ref{ex:non-surjective}, we set $\rho=10^4$ and set $\gamma$ to be larger than the threshold presented in Corollary \ref{q.e.p.p.}.
We generated the cluster path with $K=4500,4450,...,50,0$ in decreasing order.
We can see from the third row of Figure \ref{visual2} that NTL recovers true clusters for the dataset 1, but not fully for the dataset 2, in that we can see that some points joined in the opposite clusters for several small $K$'s.

Finally, we consider using NL to improve NTL.
The bottom row of Figure \ref{visual2} shows cluster paths generated by NTL starting with the initial point generated by NL. 
The midpoint in the cluster path of NL, denoted by small black points in the bottom row of Figure \ref{visual2}, was employed as the initial point for NTL.\footnote{More precisely, the midpoint was defined among the points of centroids $\{(x_1(\gamma),...,x_n(\gamma)):\gamma=10^{-3}\times(1.2)^{t-1},t=1,...,50\}$ where at least one centroid was different from one of the others.}
The choice of this initial point is motivated by the fact that samples belonging to the same true cluster are still likely to be closer to each other even if NL does not work well, as shown in the second row of Figure \ref{visual2}.
In contrast with the case where NTL is only applied, we can see that it is better classified for both data sets. 
These results support the use of NTL when no prior information is available.

 \begin{figure}[H]
  \begin{center}
  \begin{tabular}{|c|c|} \hline
   Case 1 & Case 2 \\ \hline\hline 
   \begin{minipage}[b]{0.45\linewidth}
    \centering
    \vspace{1pt}
    \includegraphics[width=1.0\columnwidth]{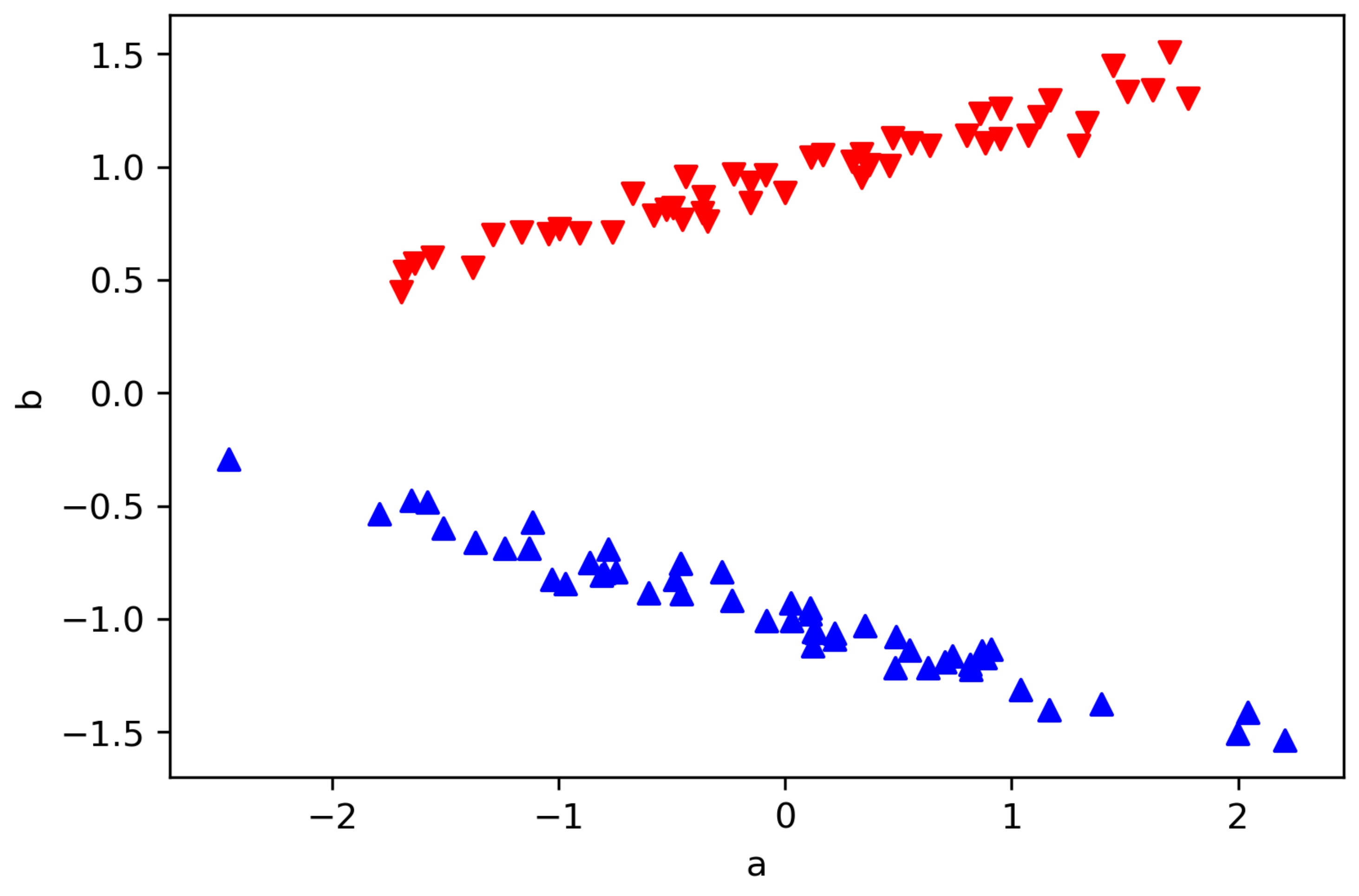}
   \end{minipage} &
   \begin{minipage}[b]{0.45\linewidth}
    \centering
    \vspace{1pt}
    \includegraphics[width=1.0\columnwidth]{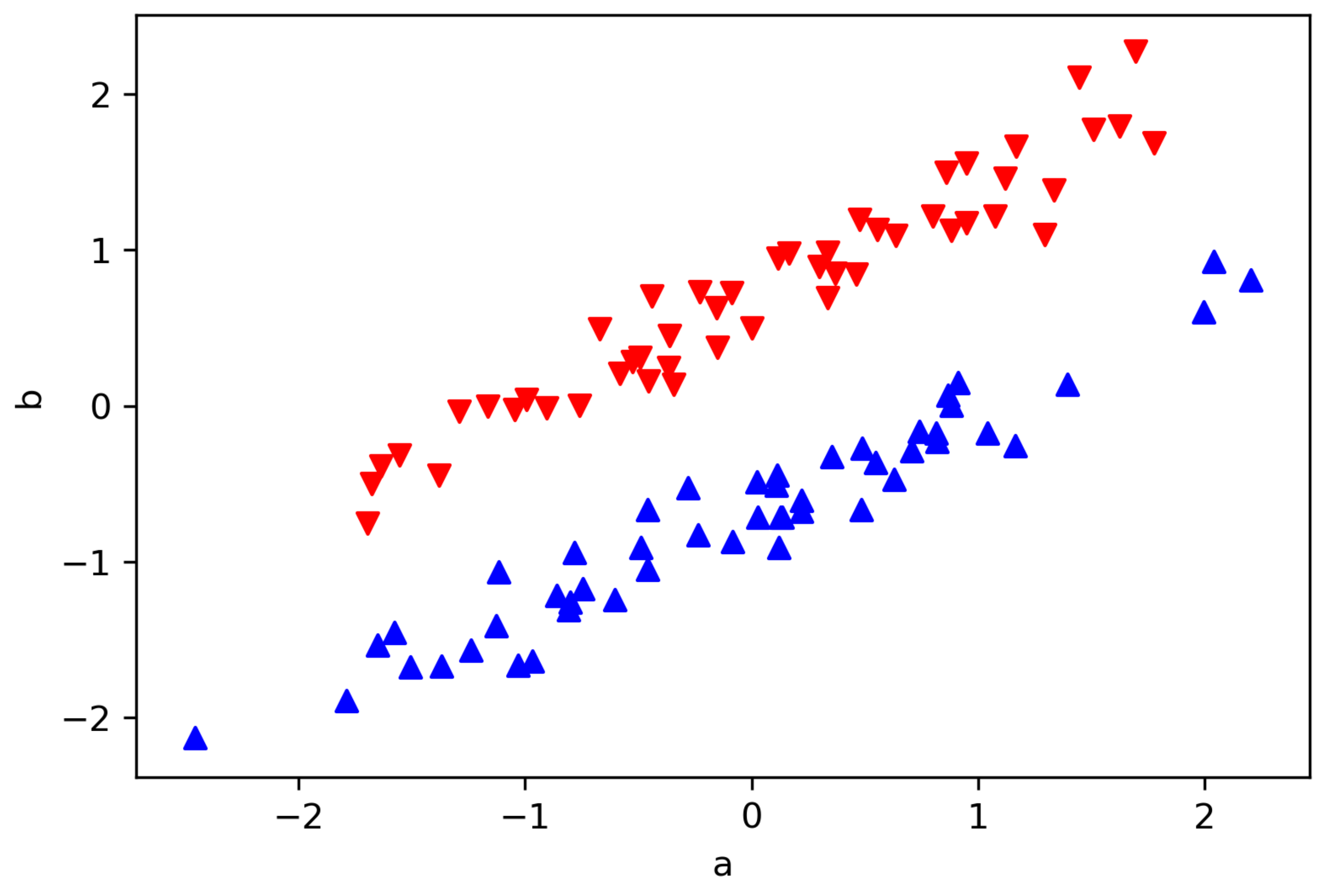}
   \end{minipage} \\ 
(a) plot of dataset & (b) plot of dataset \\
\hline 
   \begin{minipage}[b]{0.45\linewidth}
    \centering
    \vspace{1pt}
    \includegraphics[width=1.0\columnwidth]{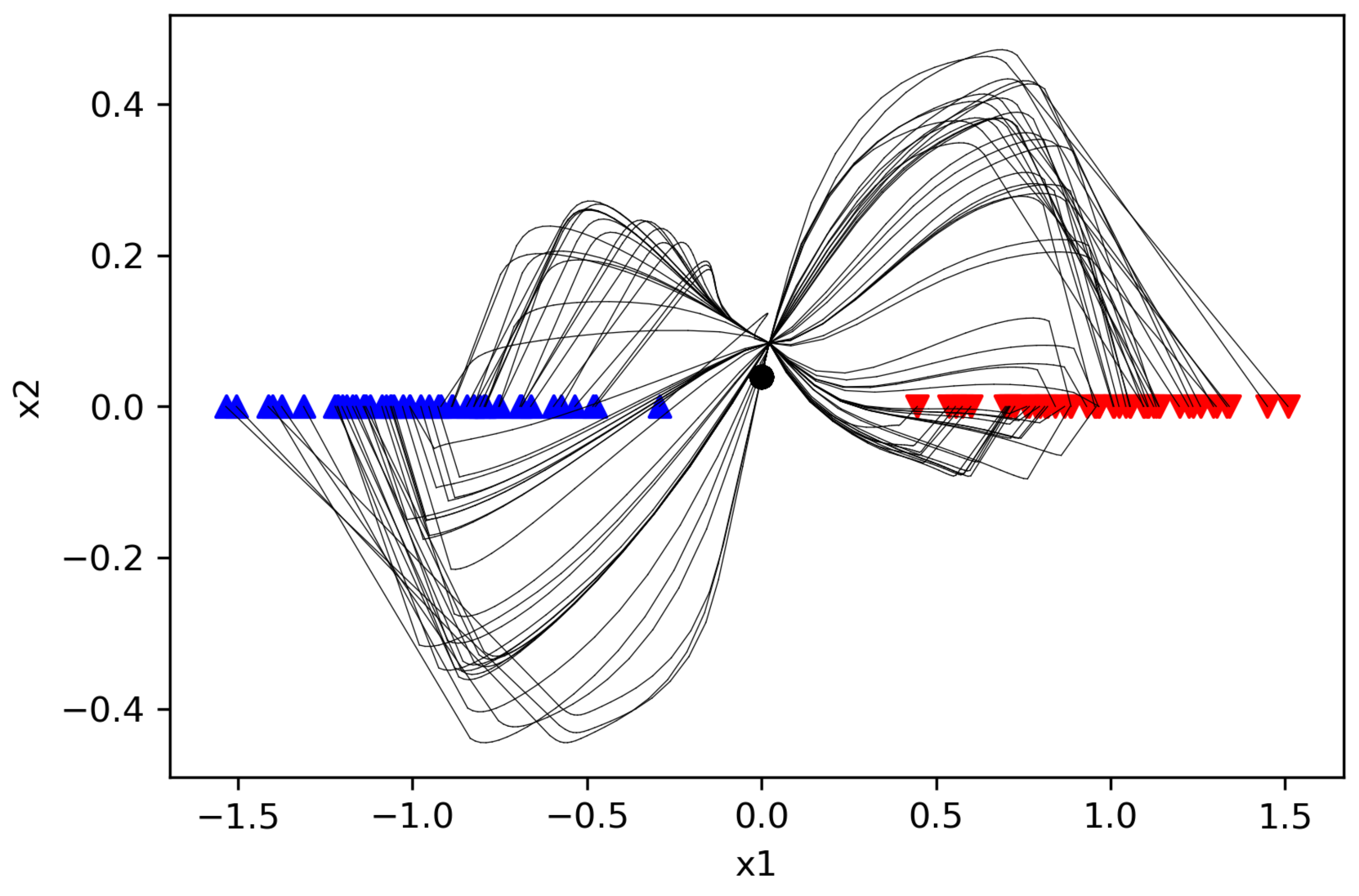}
   \end{minipage} &
   \begin{minipage}[b]{0.45\linewidth}
    \centering
    \vspace{1pt}
    \includegraphics[width=1.0\columnwidth]{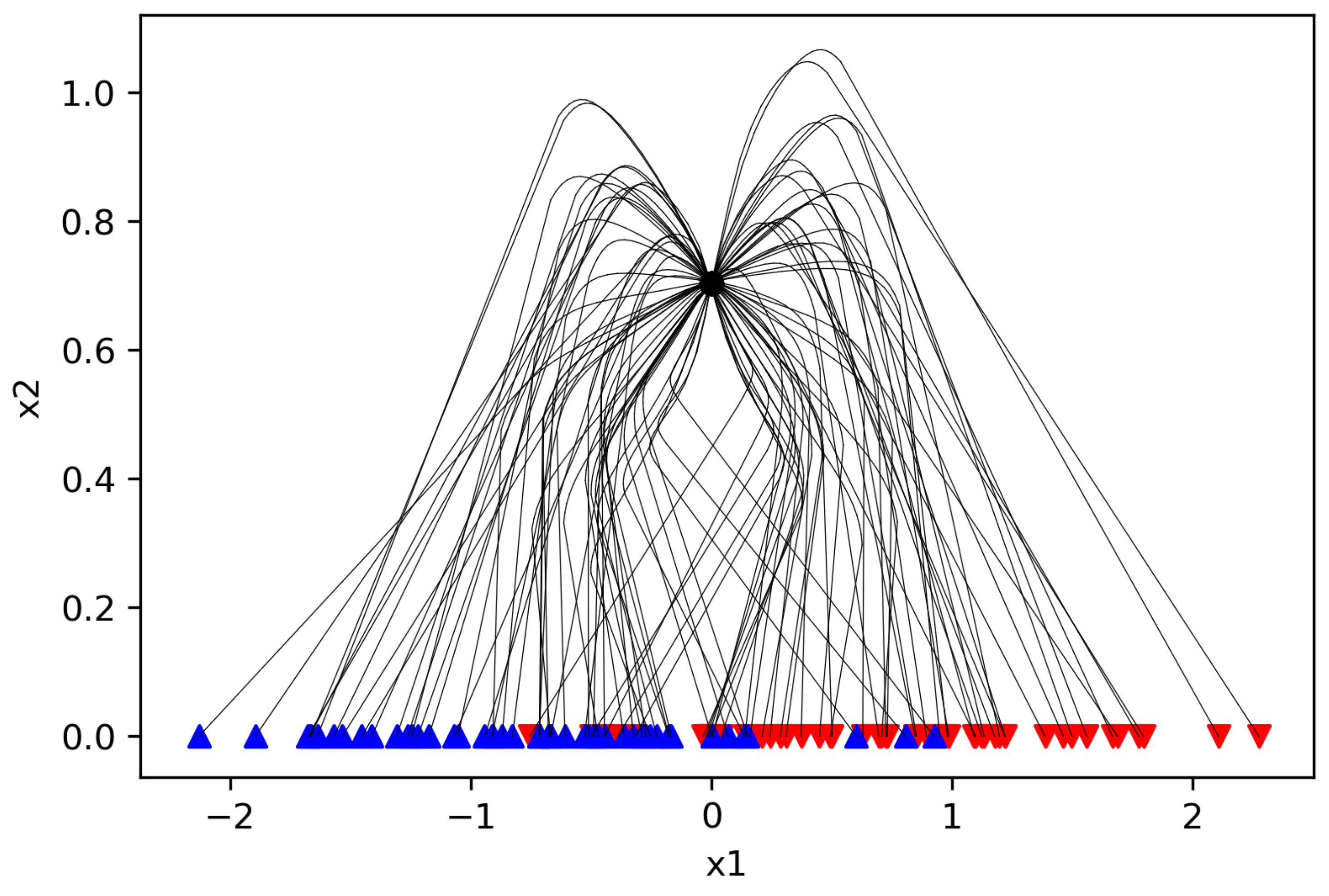}
   \end{minipage} \\ 
(c) NL & (d) NL \\
\hline 
   \begin{minipage}[b]{0.45\linewidth}
    \centering
    \vspace{1pt}
    \includegraphics[width=1.0\columnwidth]{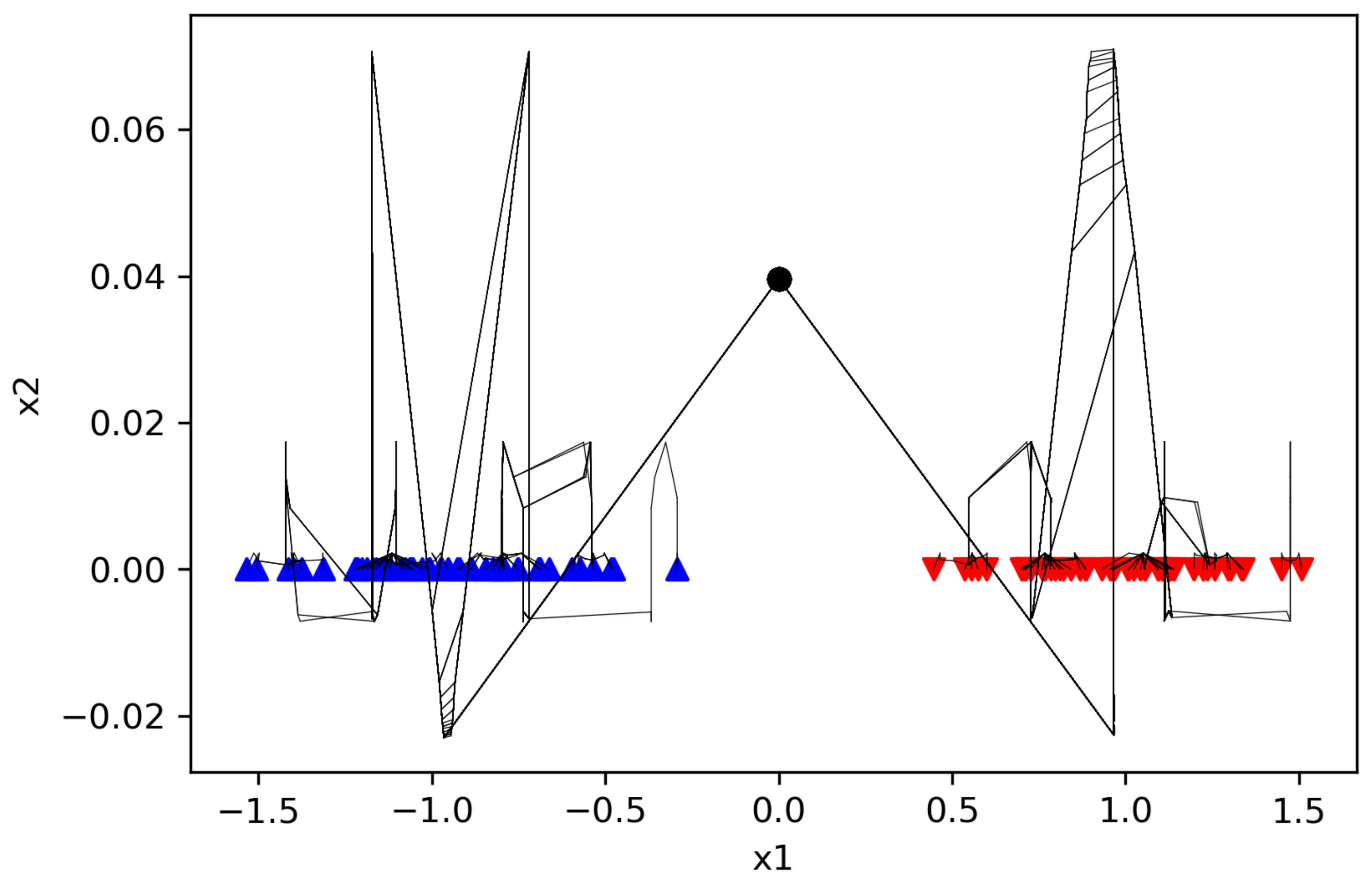}
   \end{minipage} &
   \begin{minipage}[b]{0.45\linewidth}
    \centering
    \vspace{1pt}
    \includegraphics[width=1.0\columnwidth]{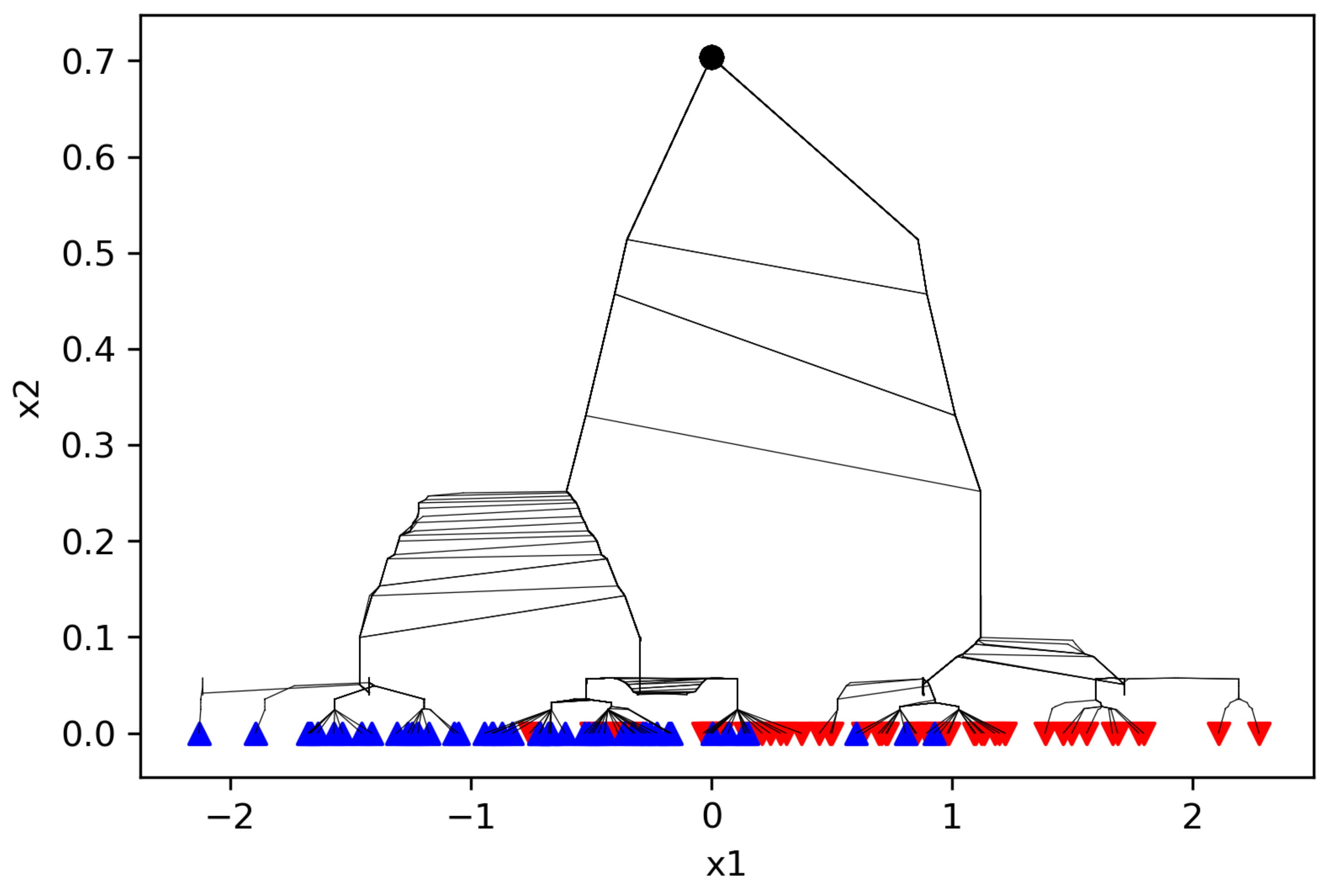}
   \end{minipage} \\ 
(e) NTL & (f) NTL \\
\hline 
   \begin{minipage}[b]{0.45\linewidth}
    \centering
    \vspace{1pt}
    \includegraphics[width=1.0\columnwidth]{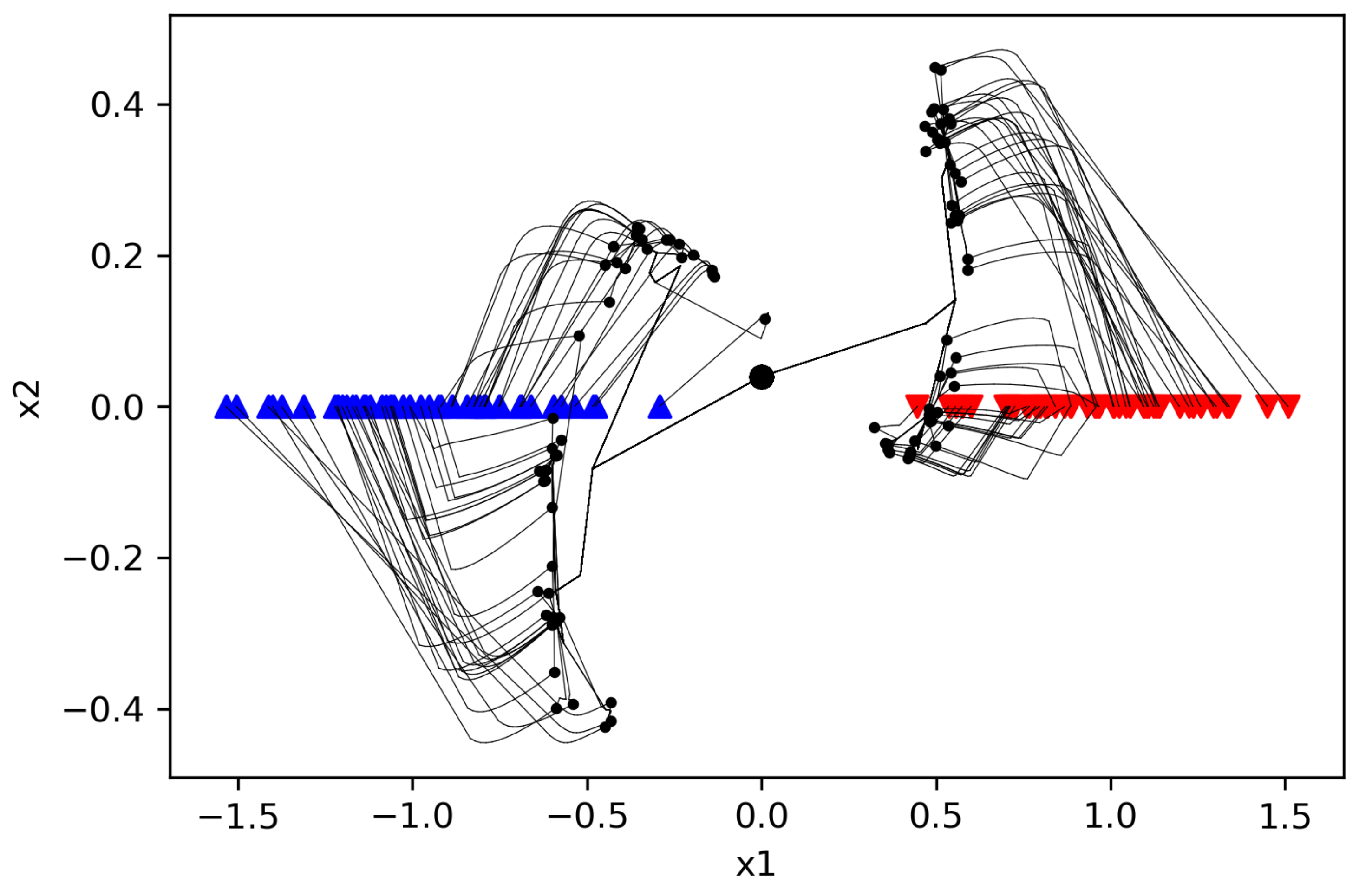}
   \end{minipage} &
   \begin{minipage}[b]{0.45\linewidth}
    \centering
    \vspace{1pt}
    \includegraphics[width=1.0\columnwidth]{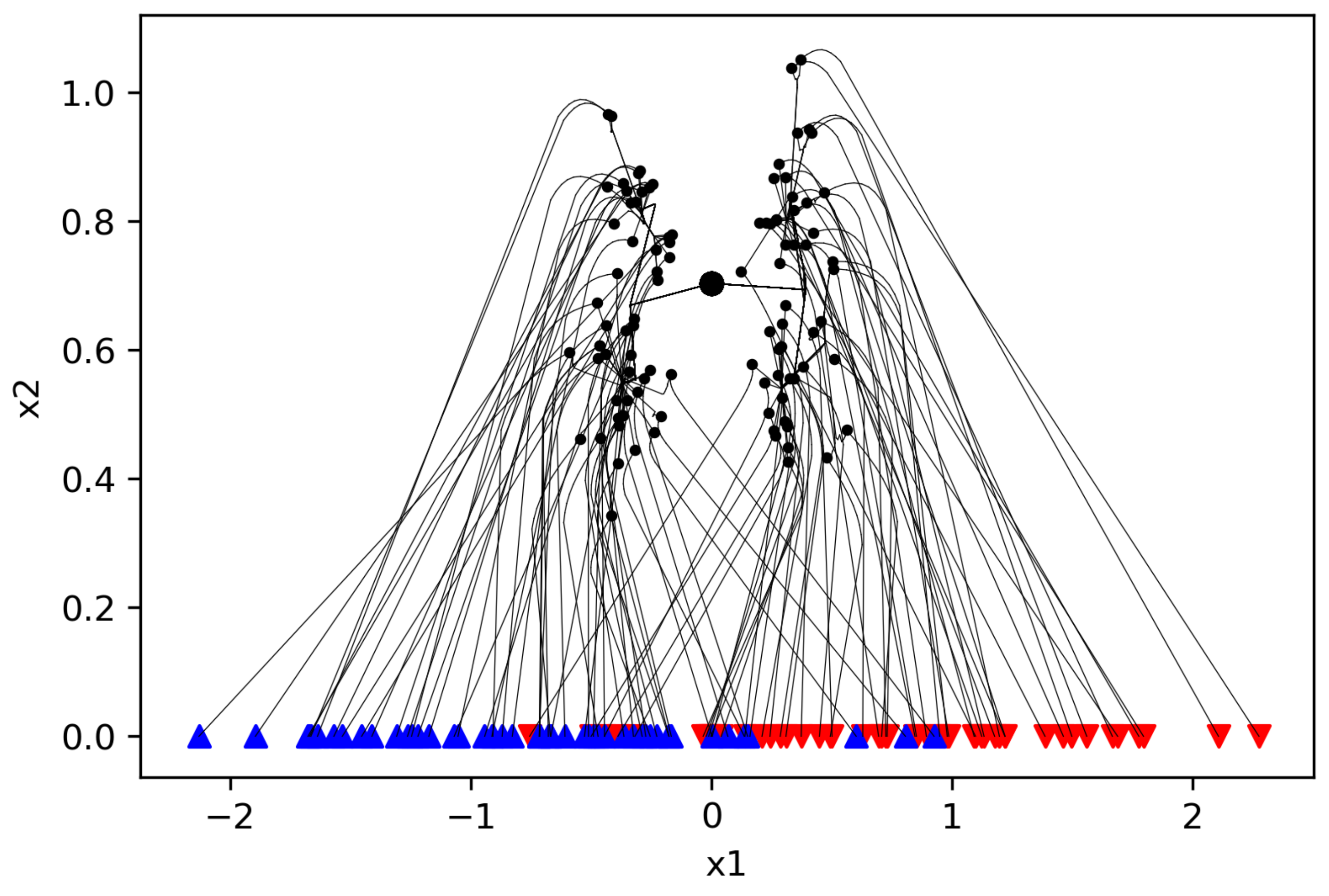}
   \end{minipage} \\ 
(g) NL + NTL & (h) NL + NTL \\
\hline
  \end{tabular}
  \caption{Datasets and cluster paths of centroids for simple regression.}
  \label{visual2}
  \end{center}
 \end{figure}

 \subsection{Ordinary clustering problem} 
This subsection compares \eqref{dcclustering} in Example \ref{dcc} with Convex Clustering (CC for short). 
Namely, we set $f_i(x_i)=\frac{1}{2}\|x_i-a_i\|_2^2$ and $\mathcal{E}=\left\{\{i,j\}\mid i\neq j, i,j\in\mathcal{V}\right\}$ in NL \eqref{nlasso} and NTL \eqref{dcnl}.
As mentioned in Section 1, in this case, we have access to prior information.

Firstly, we consider half moons dataset.
The used dataset ($n=100, p=2$) is shown in the upper left corner of Figure \ref{visual3}, where given (true) cluster labels of the data points are indicated by different colors (red versus blue).
As for the weights for CC, we consider two cases: (i) uniform weights, $w_{\{i,j\}}=1$ for all $\{i,j\}\in\mathcal{E}$, and (ii) non-uniform weights.
In order to define non-uniform weights for case (ii), let us denote the $k$-nearest neighbors of a point $i\in\mathcal{V}$ by
\begin{align}
\mathrm{NN}(i, k):=\{j\in\mathcal{V}\mid a_j\ \mathrm{is\ one\ of}\ k\ \mathrm{nearest\ neighbors\ of}\ a_i\}.
\end{align}
With this, we define 
\begin{align}
w_{\{i,j\}}=
\left\{
\begin{array}{ll}
\exp(-0.5{\|a_i-a_j\|}_2^2), & \mbox{if }i\in \mathrm{NN}(j, 20)\ \mathrm{or}\ j\in \mathrm{NN}(i, 20),\\
0,                                  & \mbox{otherwise},
\end{array}
\right.
\end{align}
for $\{i,j\}\in\mathcal{E}$.

For NL, a cluster path of centroids for $\gamma\in\{10^{-3}\times2^{t-1}\}_{t=1}^{50}$ is computed with initial points $\bm{y}^0=0, x_i^0=a_i ~(i\in \mathcal{V})$.\footnote{When all centroids degenerate at a single point, the computation of the path was stopped.}
As for NTL, we used $\rho={10}^{4}$ and started from the same initial points, computing a cluster path of centroids for $K\in\{19900,19800,...,100,0\}$. From Corollary \ref{c.e.p.p.}, the penalty parameter $\gamma$ is set to be $\gamma=3n\max_i{\|a_i\|}_2\times1.001$.

We can see from Figure \ref{visual3} that the CC with uniform weight failed to form clusters until it degenerated to a single point.
On the other hand, the weighted CC and NTL succeeded in showing non-obvious clusters along the cluster paths.
Comparing with the two methods, NTL generates small clusters at the beginning of the cluster path, which is more informative than NL about the closeness of points.

 \begin{figure}[H]
  \begin{center}
  \begin{tabular}{cc}
   \begin{minipage}[t]{0.45\linewidth}
    \centering
    \includegraphics[width=1.0\columnwidth]{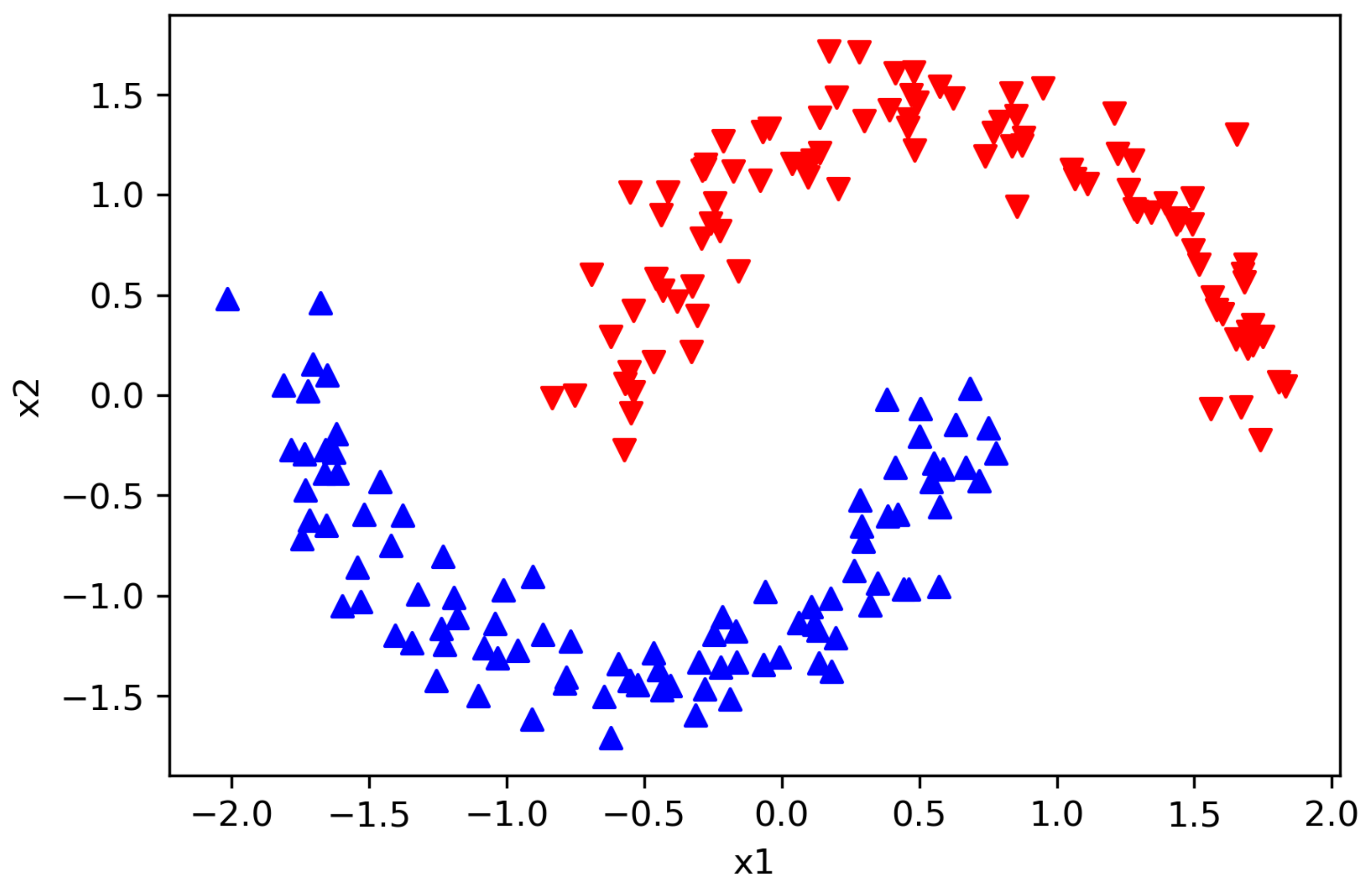}
   \end{minipage} &
   \begin{minipage}[t]{0.45\linewidth}
    \centering
    \includegraphics[width=1.0\columnwidth]{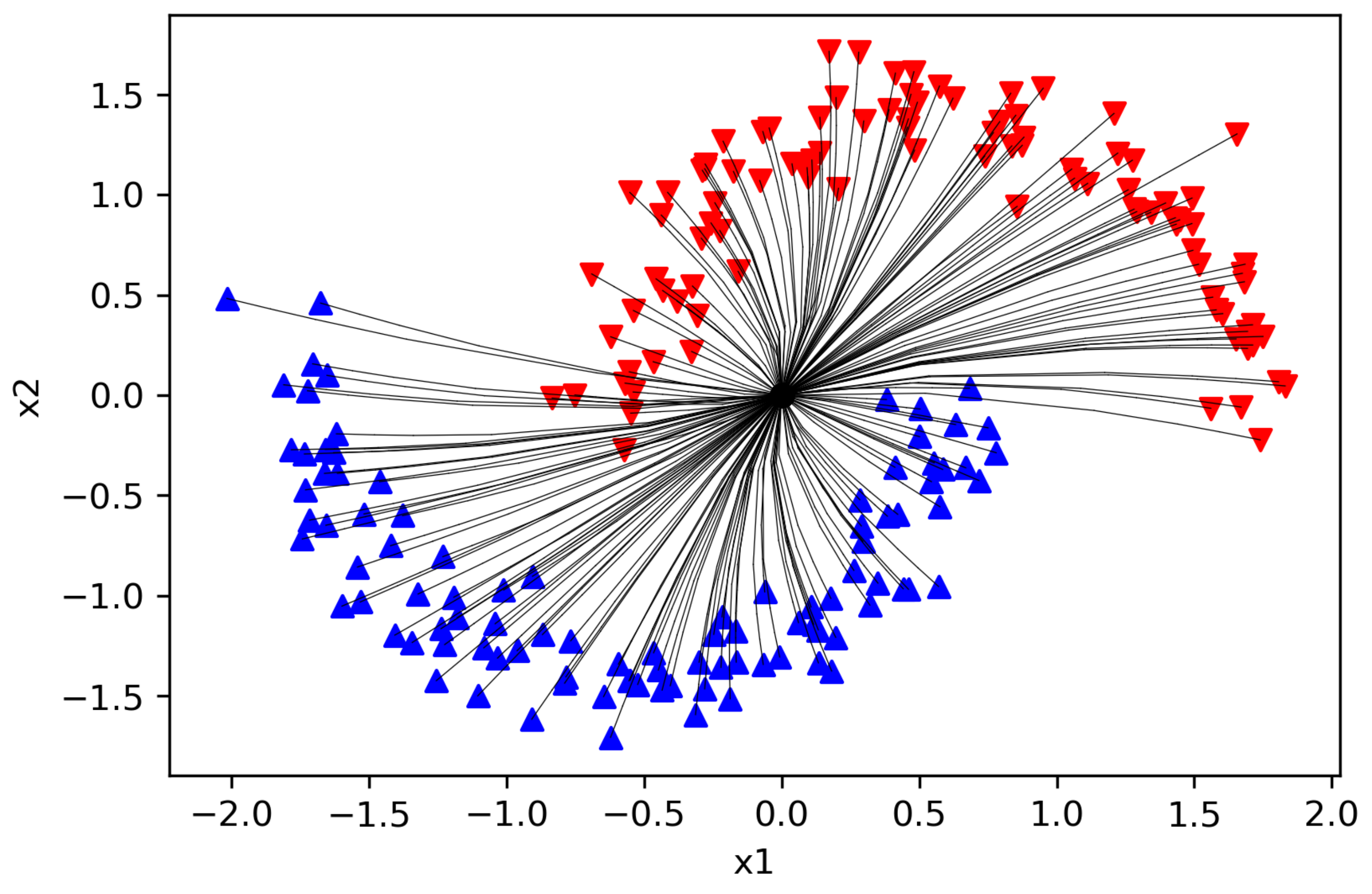}
   \end{minipage} \\
   (a) plot of dataset & (b) CC with uniform weight \\
   \begin{minipage}[t]{0.45\linewidth}
    \centering
    \includegraphics[width=1.0\columnwidth]{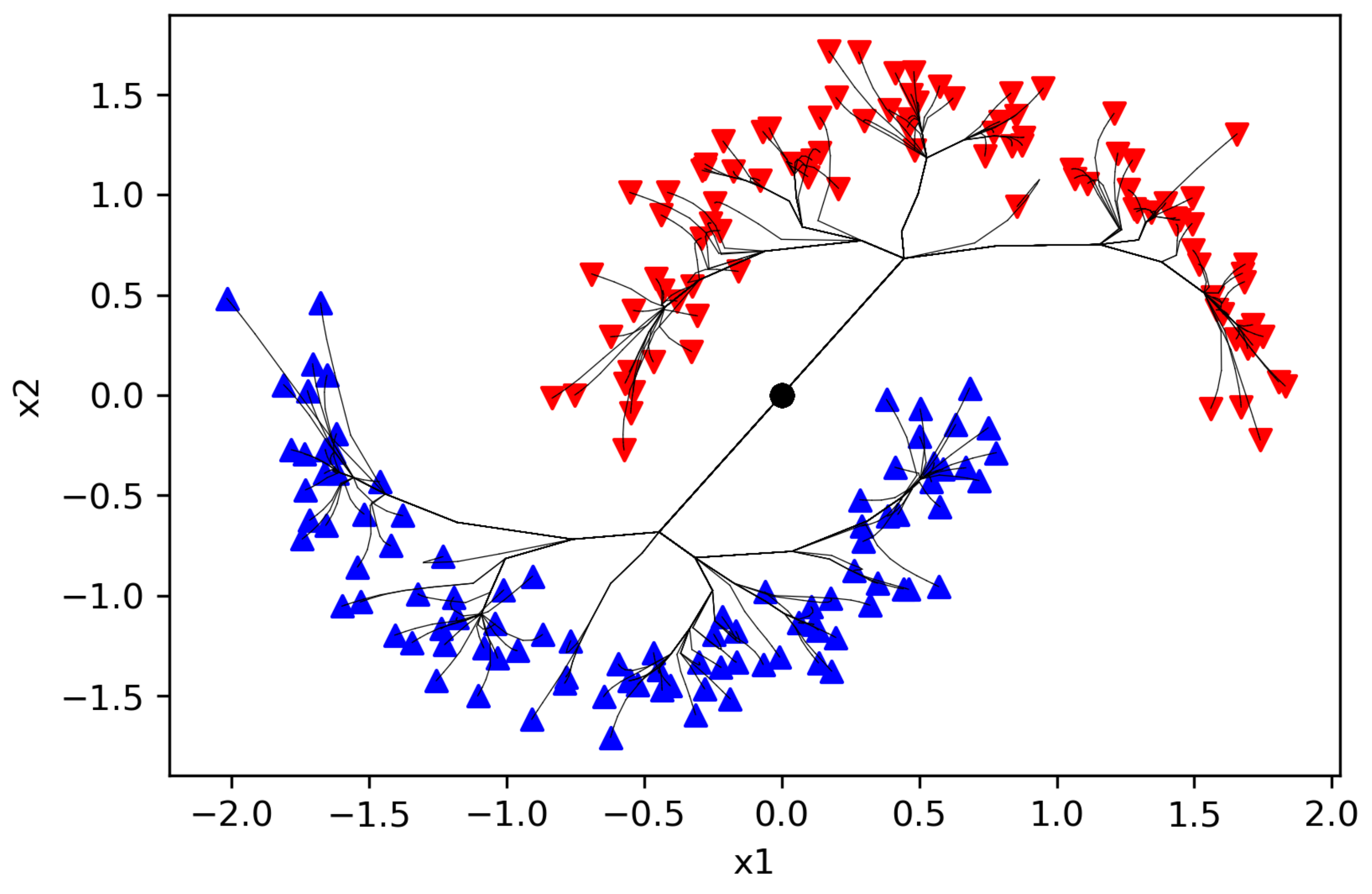}
   \end{minipage} &
   \begin{minipage}[t]{0.45\linewidth}
    \centering
    \includegraphics[width=1.0\columnwidth]{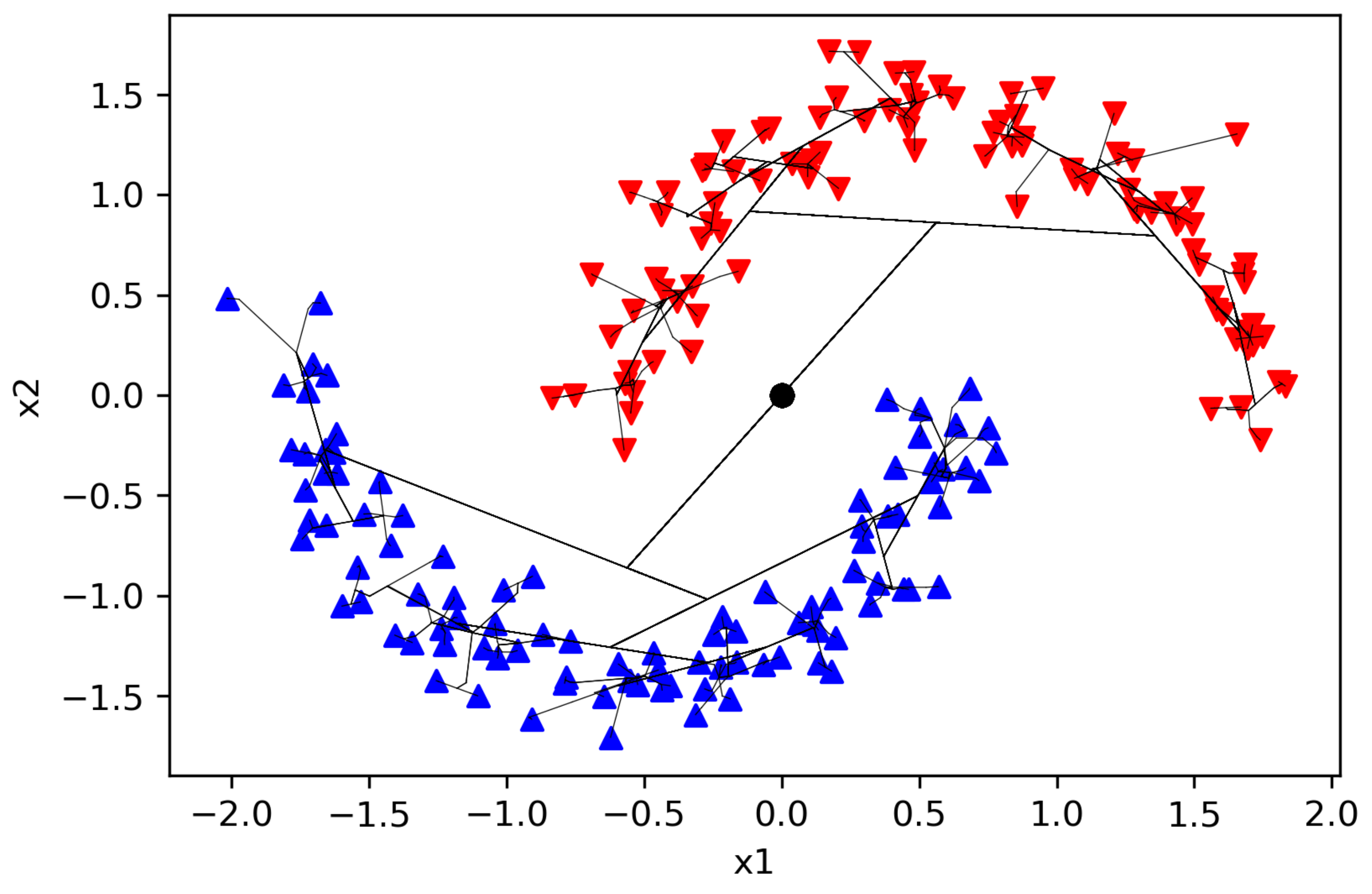}
   \end{minipage} \\
   (c) CC with Gaussian weight & (d) NTL \\
  \end{tabular}
  \caption{Datasets and cluster path of centroids for half moons.}
  \label{visual3}
  \end{center}
 \end{figure}

Next, using several real data sets \footnote{Datasets from scikit-learn https://scikit-learn.org/stable/datasets/index.html. The digit dataset was resampled so that $n=500,100,50$.}, we quantitatively compared the quality of clustering on the basis of Adjusted Rand Index \citep{lawrence1985comparing} (ARI for short; see, e.g., \citet[Section 2]{nguyen2010information} for the details). ARI takes a value between 0 and 1, and when it is closer to 1, the clustering performance is considered to be higher.

For weights for CC, we consider the following two cases:
\begin{align}
    w^1_{\{i,j\}}=
\left\{\begin{array}{ll}
\exp(-0.5{\|a_i-a_j\|}_2^2), & \mbox{if }i\in \mathrm{NN}(j, \left\lceil\frac{n}{2}\right\rceil) \mbox{ or } j\in \mathrm{NN}(i, \left\lceil\frac{n}{2}\right\rceil),\\
0,& \mbox{otherwise}
\end{array}\right.
\end{align}
and 
\begin{align}
w^2_{\{i,j\}}=
\left\{\begin{array}{ll}
\exp(-0.5{\|a_i-a_j\|}_2^2), & \mbox{if }i\in \mathrm{NN}(j, \left\lceil\frac{n}{10}\right\rceil)\ \mathrm{or}\ j\in \mathrm{NN}(i, \left\lceil\frac{n}{10}\right\rceil),\\
0,& \mbox{otherwise},
\end{array}\right.
\end{align}
where $\lceil l \rceil$ denotes the smallest integer no less than $l$. 
Note that $(w^2_{\{i,j\}})_{\{i,j\}\in\mathcal{E}}$ put more zeros on edges than $(w^1_{\{i,j\}})_{\{i,j\}\in\mathcal{E}}$.
Other settings are the same as in the previous (half-moon) example.

Table \ref{visual4} summarizes the largest values of ARI along the cluster paths.
We see from the table that weighted CC with $w^2$ performed best for three data sets, as Theorem \ref{thm:clusterrecover} implies.
On the other hand, NTL recorded the best performance with the two datasets. While it cannot be said that one is better than the other, it is clear from this experiment that if prior information is not available, CC results in poor performance.

 \begin{table}[H]
  \caption{Maximum Adjusted Rand Index through the cluster path.}
  \label{visual4}
  \begin{tabular}{|l||l|l|l|l|l|} \hline
   & iris & wine & digits ($n=500$) & digits ($n=100$) & digits ($n=50$) \\ \hline\hline
   CC (uniform)& 0.0015 & 0.0000 & 0.0082 & 0.0014 & 0.0013 \\ \hline
   CC ($w^1$) & 0.5681 & 0.7577 & 0.5302 & 0.4577 & 0.3812 \\ \hline
   CC ($w^2$) & 0.5681 & 0.7994 & \textbf{0.5346} & \textbf{0.5101} & \textbf{0.4443} \\ \hline
   NTL & \textbf{0.5778} & \textbf{0.8260} & 0.3967 & 0.4134 & 0.4207 \\ \hline
  \end{tabular}\\
\footnotesize 
The best value for each data set is shown in boldface. 
 \end{table}

\subsection{Piecewise constant fitting}
As the final example, we consider the problem of recovering a piecewise constant signal from a noisy signal \cite[Example 9.16]{calafiore2014optimization} by using NL and NTL.
Specifically, we consider a situation where $n=1000$ noisy signals $\hat{x}_1,...,\hat{x}_{1000}\in\mathbb{R}$ are generated as $\hat{x}_i=x_i^{\text{o}}+e_i$ with $x^{\rm o}$ being given as the original signal given as the black stepwise function in Figure~\ref{visual5} and $e_i$ being independently drawn from a normal distribution $N(0,{0.2}^2)$.
Given the time series structure, we set $\mathcal{V}=\{1,...,1000\}$, $\mathcal{E}=\left\{\{i,i+1\}\mid i,i+1\in\mathcal{V}\right\}$, $f_i(x_i)=\frac{1}{2}{(x_i-\hat{x}_i)}^2$, and
\begin{align}
D:=\left(
\begin{array}{ccccc}
1 & -1 & 0 & \cdots & 0\\
0 & 1 & -1 & \ddots & \vdots\\
\vdots & \ddots & \ddots & \ddots & 0\\
0 & \cdots & 0 & 1 & -1
\end{array}
\right).
\end{align}
It is known that $\sigma:=\lambda_{\min}(DD^\top)=2(1-\rm{cos}(\frac{\pi}{1000}))\approx 9.87\times 10^{-6}$ \citep[Theorem 2.2]{kulkarni1999eigenvalues}.
In this example, we consider not only the perspective of the cluster recovery but also the quality of the solution.
The quality of the solution here means the closeness of the recovered signal and the original signal, i.e., $\|\bm{x}^*-\bm{x}^\text{o}\|_2$.

For NL, we computed the cluster path of centroids for $\gamma\in\{10^{-3}\times(1.2)^{t-1}\}_{t=1}^{100}$ with the initial points $\bm{y}^0=0, x_i^0=\hat{x}_i$ and the prior information $w_{\{i,i+1\}}=
\exp(-0.5{\|\hat{x}_i-\hat{x}_{i+1}\|}_2^2)$. 
As for NTL, we applied ADMM from the same initial point, using $K=5$ and $\gamma=3n\max_i{\|\hat{x}_i\|}_2\times1.001$. 
As an alternative heuristic for the cluster path algorithm, the initial value of $\rho$ was set to $1$ though Examples \ref{ex:ordinary} and \ref{ex:boundness} suggest setting as $\rho>\max\{\frac{2}{\sigma},\frac{1}{\sigma}\}=\frac{2}{\sigma}>2\times 10^5$. 
The parameter was updated by the formula $\rho\leftarrow\min\{10\rho,\frac{2}{0.99\sigma}\}$ every 100 iterations. 
This modification is motivated by a similar heuristic used in \citet{li2015global}.

Figure~\ref{visual5} shows how well NL and NTL recover the original signal, which is denoted by black solid line, from the noisy signal, which is shown by the red solid line in the upper left panel.
Since there is a degree of freedom in the evaluation criteria, two best-case results are given for NL. 
The panel (c) is the best in solution quality in the sense that the smallest value of $\|\bm{x}^*-\bm{x}^\text{o}\|_2$ was attained out of 100 values of $\gamma$. On the other hand, the panel (d) is the best in the cardinality in the sense that the employed $\gamma$ is the smallest out of the 100 values such that the number of jumped points is less than 5, which is the number of jumps in the original signal.
We see from the panel (d) of Figure~\ref{visual5} that NL detected the jump points almost exactly as Theorem~\ref{thm:clusterrecover} implies, but the levels of the piecewise constants are far from the original signal. 
We think this is due to the fact that the degree of each node is at most 2, so that prior information was not given enough to recover the signal by NL. 
Employing more zero-weights as prior information worked better in the experiment of the previous subsection, but this example indicates that that is not always true. 
This indicates that it is not easy to give weights adequately for NL in advance.
On the other hand, NTL not only detects the jump accurately but also estimates the levels of the piecewise constant more accurately than the best case of NL (lower left panel).

 \begin{figure}[H]
  \begin{center}
  \begin{tabular}{cc}
   \begin{minipage}[t]{0.45\linewidth}
    \centering
    \includegraphics[width=1.0\columnwidth]{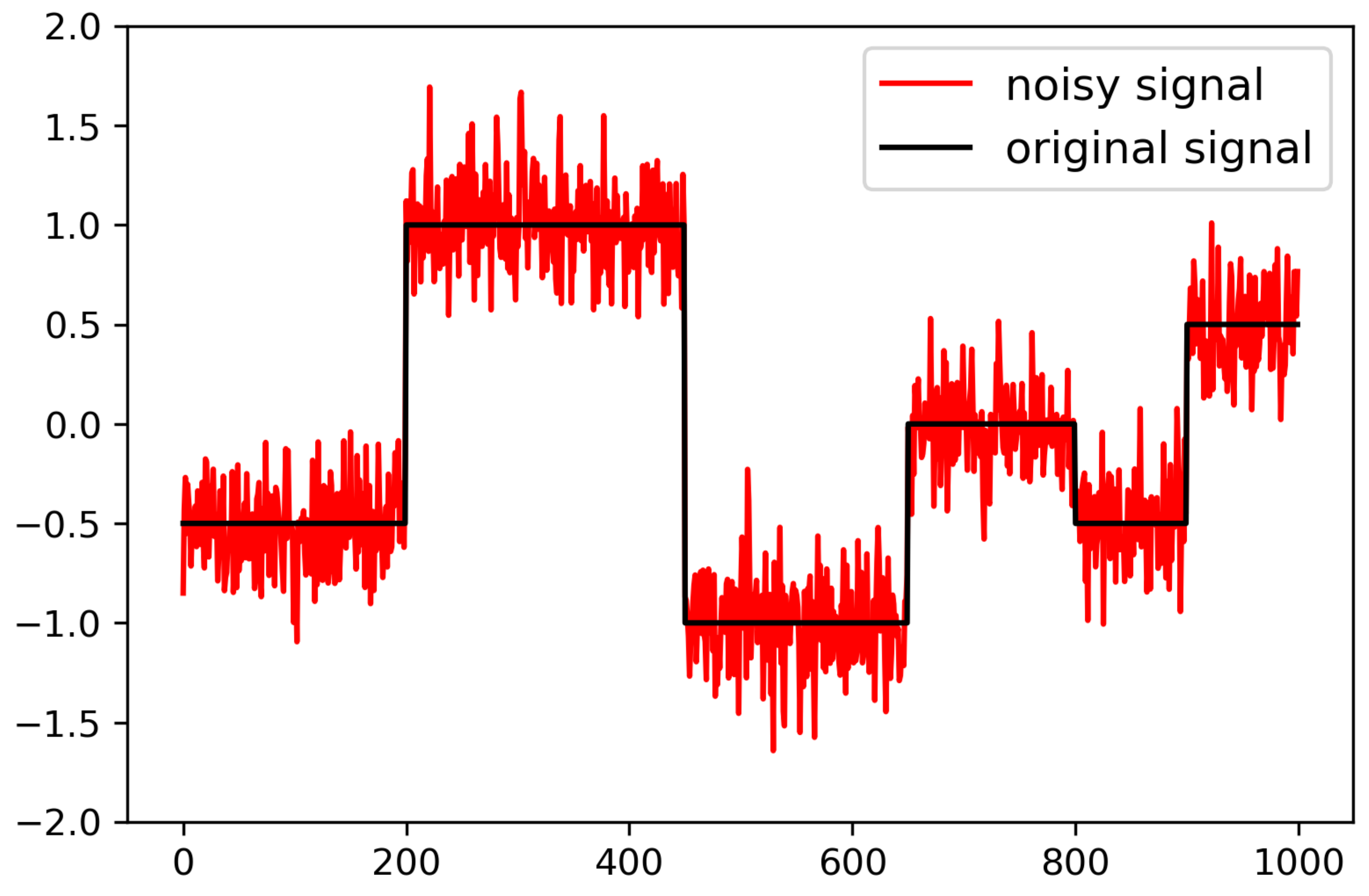}
   \end{minipage} &
   \begin{minipage}[t]{0.45\linewidth}
    \centering
    \includegraphics[width=1.0\columnwidth]{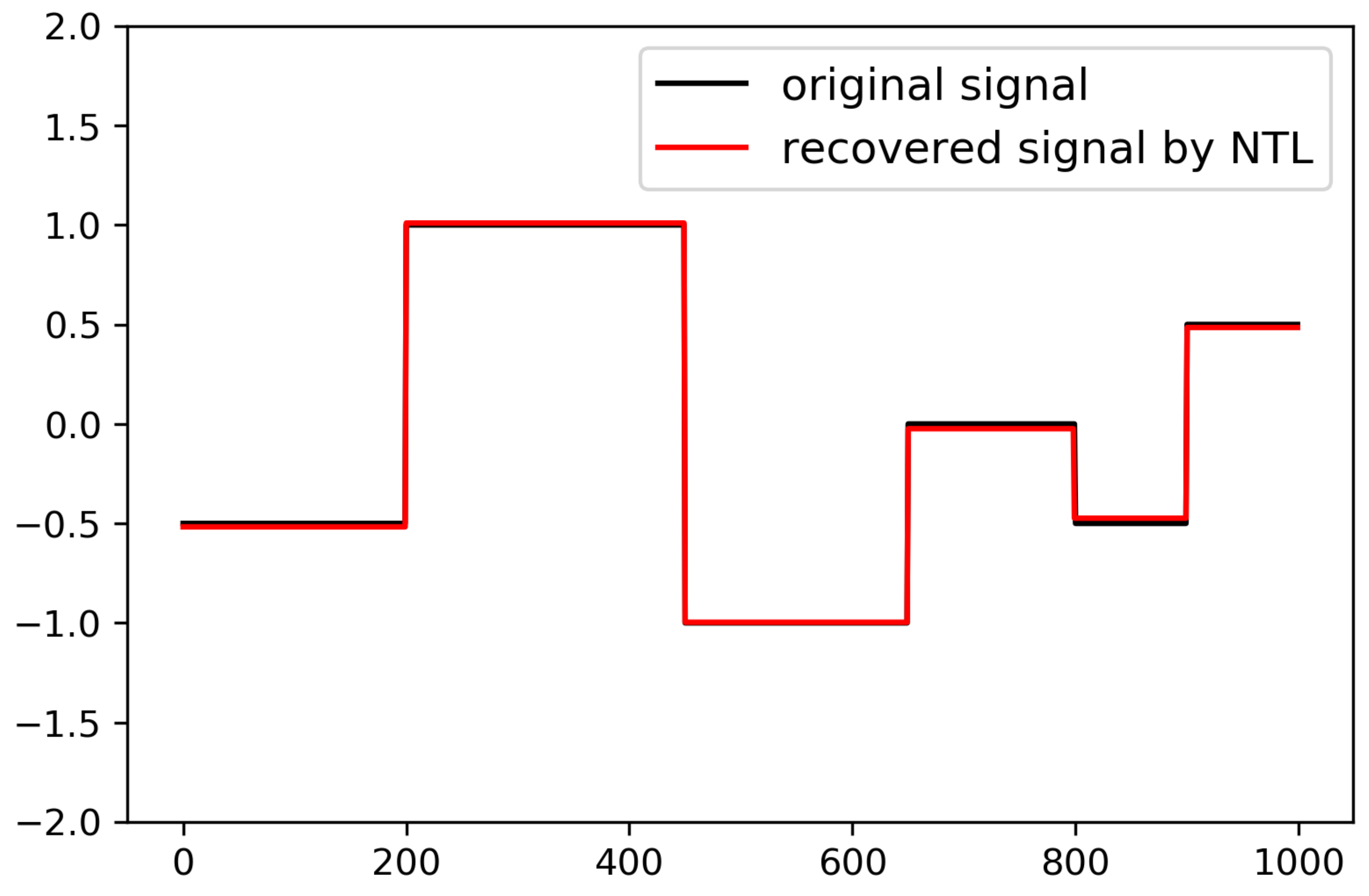}
   \end{minipage} \\
   (a) plot of dataset & (b) NTL \\
   \begin{minipage}[t]{0.45\linewidth}
    \centering
    \includegraphics[width=1.0\columnwidth]{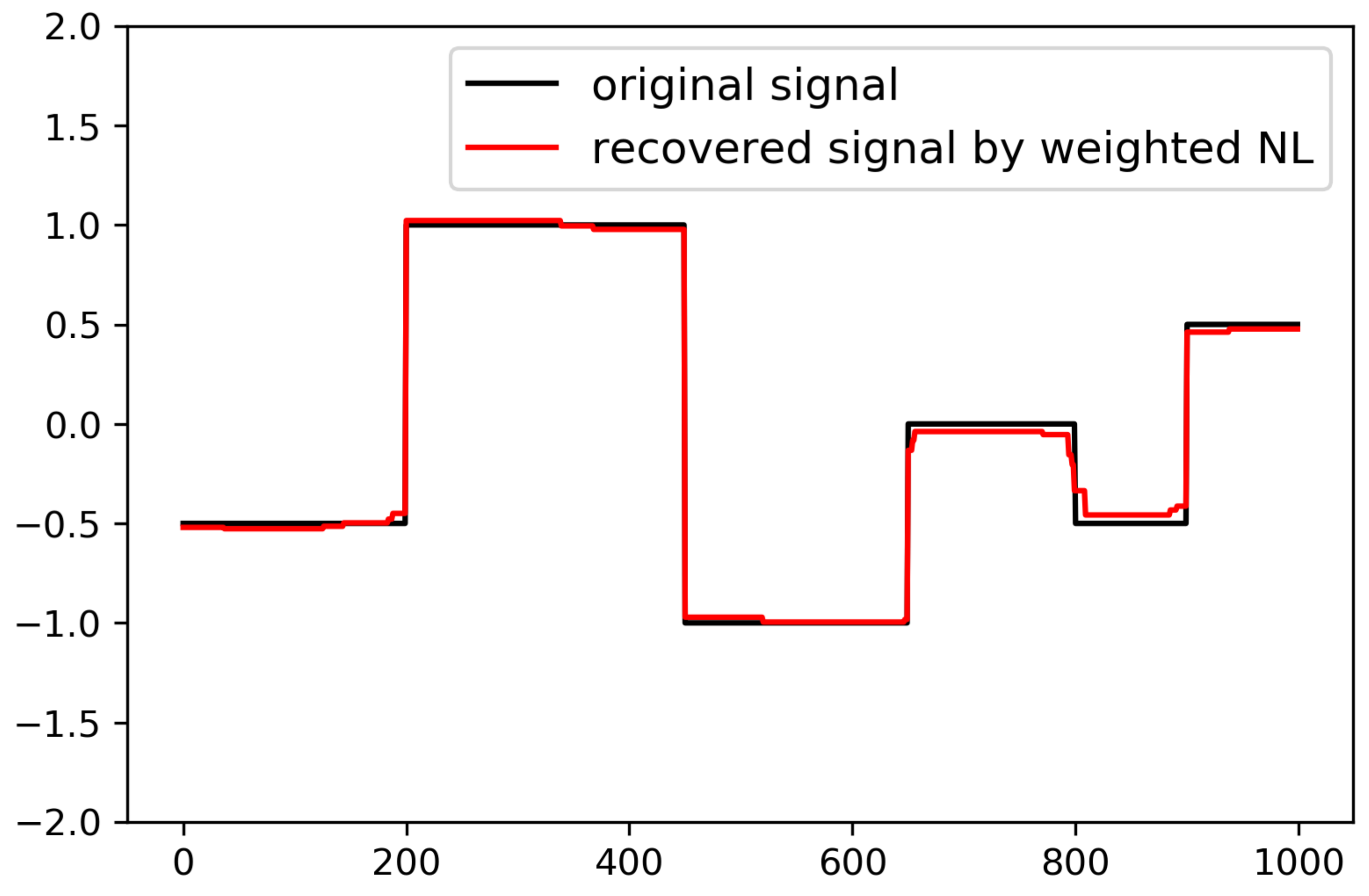}
   \end{minipage} &
   \begin{minipage}[t]{0.45\linewidth}
    \centering
    \includegraphics[width=1.0\columnwidth]{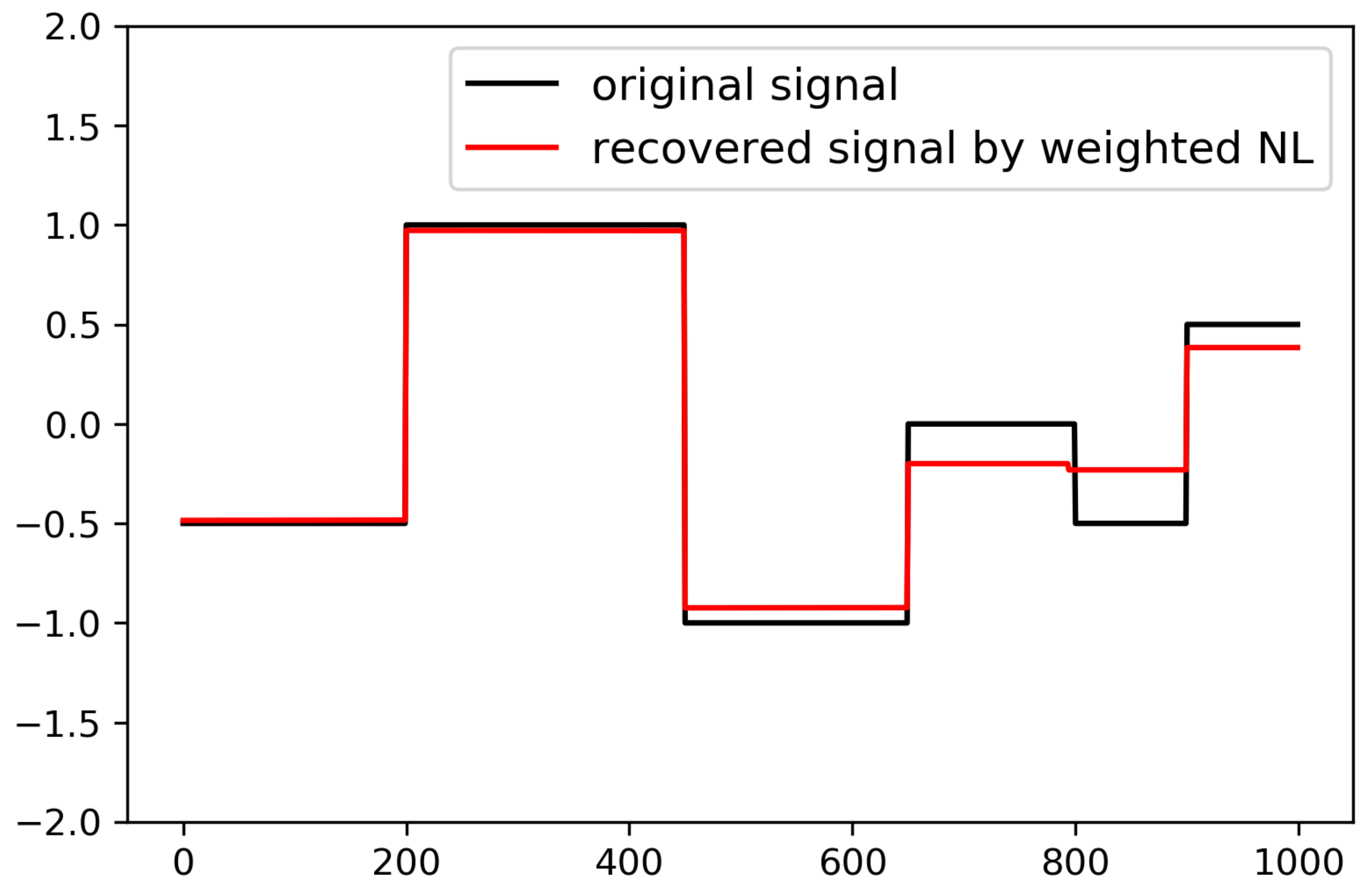}
   \end{minipage} \\
   (c) NL (best in terms of quality) & (d) NL (best in terms of cluster recovery) \\
  \end{tabular}
  \caption{Noisy signal and recovered signals.}
  \label{visual5}
  \end{center}
 \end{figure}

\section{Concluding remarks}
\label{sec:conclusion}
In this paper, we study the cluster structure of Network Lasso (NL) from a couple of different angles. 
First, we derive a condition under which NL can recover (unseen) true clusters. 
Second, to obtain clusters that might not be attained by NL, we consider a cardinality-constraint on the number of unmerged pairs of centroids and show an equivalent unconstrained reformulation called Network Trimmed Lasso (NTL). 
Numerical examples demonstrate how NTL performs better than the ordinary NL, especially when any prior information is not available.
These results suggest that we should use NL if we are given sufficient prior information, and use NTL otherwise.
We also show the convergence of ADMM to a locally optimal solution of NTL or  the cardinality-constrained problem.
However, when the underlying graph is dense and large, the employed algorithm based on ADMM would lead to an impractical solution time. For example, when the graph is a complete graph, i.e., $\mathcal{E}=\left\{\{i,j\}\mid i\neq j, i,j\in\mathcal{V}\right\}$, ADMM has to deal with $\frac{n(n-1)}{2}p$-dimension vectors, which would be prohibitively large even for a moderate $n$, say, $n=1000$. 
Developing an efficient algorithm for such large data sets is left for future research.

\section*{Acknowledgments}
Jun-ya Gotoh is supported in part by JSPS KAKENHI Grant 19H02379, 19H00808, and 20H00285.

\bibliography{reference.bib}
\bibliographystyle{plainnat}

\end{document}